\documentclass[leqno,12pt]{amsart} 
\usepackage{amsmath, amssymb, mathscinet}
\usepackage{mathtools}
\usepackage{amsthm} 
\usepackage[all]{xy}
\usepackage{enumitem}
\usepackage[T1]{fontenc}
\usepackage{hyperref} 
\usepackage{mathrsfs} 
\usepackage{graphicx}
\usepackage{cleveref}
\usepackage[svgnames]{xcolor}
\usepackage{tikz}
\usepackage{pagecolor}
\theoremstyle{plain} 
\newtheorem{theorem}{\indent\sc Theorem}[section]
\newtheorem{lemma}[theorem]{\indent\sc Lemma}
\newtheorem{corollary}[theorem]{\indent\sc Corollary}
\newtheorem{proposition}[theorem]{\indent\sc Proposition}
\newtheorem{claim}[theorem]{\indent\sc Claim}
\theoremstyle{definition} 
\newtheorem{definition}[theorem]{\indent\sc Definition}
\newtheorem{remark}[theorem]{\indent\sc Remark}
\newtheorem{example}[theorem]{\indent\sc Example}

\newcommand{\R}{\mathbb{R}}

\newcommand{\N}{\mathbb{N}}

\def\dim{\mathop{\mathrm{dim}}\nolimits}

\def\Im{\mathop{\mathrm{Im}}\nolimits}

\newcommand{\abs}[1]{\left\lvert#1\right\rvert}

\newcommand{\diam}{\mathrm{diam}}

\newcommand{\ngamma}{\mathtt{n}\gamma}

\crefname{theorem}{Theorem}{Theorems}
\crefname{lemma}{Lemma}{Lemmata}
\crefname{corollary}{Corollary}{Corollaries}
\crefname{proposition}{Proposition}{Propositions}
\crefname{remark}{Remark}{Remarks}
\crefname{example}{Example}{Examples}
\crefname{section}{Section}{Sections}
\crefname{definition}{Definition}{Definition}
\crefname{equation}{}{}
\crefname{claim}{Claim}{Claims}
\newcommand{\calO}{\mathcal{O}}

\newcommand{\OdX}{\calO \partial X}

\newcommand{\rp}[1]{\textup{\texttt{rp}}#1}
\newcommand{\Rp}{\R_{\geq 0}}
\newcommand{\ds}[1]{d_X(#1)}

\newcommand{\DA}{\Omega}
\newcommand{\rest}{\!\!\upharpoonright}
\newcommand{\rpgm}{\rp{\bar{\gamma}}}
\newcommand{\starNb}{\mathsf{star}}
\newcommand{\augn}[2]{A_{#2}(#1)}
\newcommand{\aug}[1]{A(#1)}
\newcommand{\Horo}{\mathcal{H}}
\newcommand{\cbHoro}{\mathcal{H}_\mathsf{comb}}
\newcommand{\limone}{\varprojlim{\!}^1}
\newcommand{\cd}{\operatorname{cd}}

\usepackage{version}	
\begin{document}

\title[Boundary of free products of metric spaces]{Boundary of free products of metric spaces} 

\author[T. Fukaya]{Tomohiro Fukaya$^*$} 

\author[T. Matsuka]{Takumi Matsuka} 



\renewcommand{\thefootnote}{\fnsymbol{footnote}}
\footnote[0]{2020\textit{ Mathematics Subject Classification}.
Primary 51F30; Secondary 20F65, 58B34.}


\thanks{ 
The first author was supported by JSPS KAKENHI Grant number JP19K03471. 
The second author was supported by JST, the establishment of university fellowships towards the creation of science technology innovations, Grant number JPMJFS2139.
}
\address{ Tomohiro Fukaya \endgraf
Department of Mathematics and Information Sciences,
Tokyo Metropolitan University,
Minami-osawa Hachioji, Tokyo, 192-0397, Japan
}
\email{tmhr@tmu.ac.jp}

\address{  Takumi Matsuka \endgraf
Department of Mathematics and Information Sciences,
Tokyo Metropolitan University,
Minami-osawa Hachioji, Tokyo, 192-0397, Japan
}
\email{takumi.matsuka1@gmail.com}


\begin{abstract}
In this paper, we compute (co)homologies of ideal boundaries of
free products of geodesic coarsely convex spaces in terms of those of each of the components. The (co)homology theories we consider are,
$K$-theory, 
Alexander-Spanier cohomology, 
$K$-homology, and
Steenrod homology.
These computations led to the computation of $K$-theory of the Roe algebra of free products of geodesic coarsely convex spaces via
the coarse Baum--Connes conjecture.
\end{abstract}

\maketitle

\section{Introduction} 
Coarsely convex spaces are coarse geometric analogues of 
simply connected Riemannian manifolds of nonpositive sectional curvature
introduced by Oguni and the first author. 
In \cite{FO-CCH}, 
ideal boundaries of coarsely convex spaces are constructed, and 
it was shown that coarsely convex spaces are coarsely homotopy equivalent
to the open cones over their ideal boundaries. Combining with 
Higson-Roe's work on the coarse Baum-Connes conjecture for open 
cones \cite{MR1388312}, it was shown that the coarse Baum-Connes
conjecture holds for proper coarsely convex spaces \cite{FO-CCH}.

The class of coarsely convex spaces includes geodesic Gromov hyperbolic 
spaces, CAT(0) spaces, Busemann spaces, 
and systolic complexes~\cite{bdrySytolic}. 
Descombes and Lang~\cite{convex-bicomb} showed that proper injective
metric spaces admit geodesic convex bicombings. This implies that
proper injective metric spaces are coarsely convex. 
Recently, 
Chalopin, Chepoi, Genevois, Hirai, and Osajda~\cite{chalopin2020helly-arXiv}
showed that Helly graphs are coarsely dense in their injective hulls, 
and these injective hulls are proper. It follows that
Helly graphs are coarsely convex.


In \cite{fukaya2023free}, we constructed free products of metric spaces,
which generalizes the construction by 
Bridson-Haefliger~\cite[Theorem II.11.18]{MR1744486}.
We showed that free products of geodesic coarsely convex spaces are
geodesic coarsely convex. In \cite{gcc-pair-preparation}, 
the authors introduced trees of spaces, 
which are generalizations of free products of metric spaces.
They showed that if each of the components of a tree of spaces is 
geodesic coarsely convex, then the tree of spaces is geodesic 
coarsely convex.
This generalized the main result of \cite{fukaya2023free}
for trees of spaces.

In this paper, we study the topology of ideal boundaries of trees of
spaces whose components are coarsely convex. For this purpose, we
introduce augmented spaces of trees of spaces. We determine
the topological type of the ideal boundaries of augmented spaces.
Then we compute $K$-theory, Alexander-Spanier cohomology, 
and $K$-homology of the ideal boundaries of the trees of spaces
in terms of ideal boundaries of augmented spaces, 
and ideal boundaries of the components of the trees of spaces.

By applying \cref{thm:cohomology,thm:homology} for free products,
we have the following.
\begin{theorem}
\label{thm:cohomology_free_prod}
 Let $X$ and $Y$ be proper geodesic coarsely convex spaces with nets. 
 Suppose that the net of $X$ and that of $Y$ are coarsely dense
 in $X$ and $Y$, respectively, and, both $X$ and $Y$ are unbounded.
 Let $X * Y$ be the free product of $X$ and $Y$.
 Let $\partial X$, $\partial Y$, and $\partial (X*Y), $ denote
 the ideal boundaries of $X$, $Y$, and $X*Y$, respectively.
 Let $\mathcal{C}$ denote the Cantor space.

 Let $M^*=(M^n)_{n\in \N}$ be the $K$-theory or the 
 Alexander-Spanier cohomology. Then,
 \begin{align*}
 \tilde{M}^p(\partial (X * Y)) 
 &\cong \tilde{M}^p(\mathcal{C}) \oplus 
  \bigoplus_{\N} \tilde{M}^p(\partial X)
  \oplus \bigoplus_{\N} \tilde{M}^p(\partial Y)
 \end{align*}

 Let $M_*=(M_n)_{n\in \N}$ be the $K$-homology or the Steenrod homology.
 Then,
\begin{align*}
  \tilde{M}_p(\partial (X * Y)) 
 &\cong \tilde{M}_p(\mathcal{C}) \times
  \prod_{\N} \tilde{M}_p(\partial X)
  \times
  \prod_{\N} \tilde{M}_p(\partial Y).
\end{align*} 
Here $\tilde{M}_p$ and $\tilde{M}^p$ denote 
the reduced (co)homology corresponding to $M_p$ and $M^p$.
\end{theorem}

Let $Y$ be a metric space. We denote by $C^*(Y)$ 
the Roe algebra of $Y$.
The coarse Baum-Connes conjecture 
states that the $K$-theory of the Roe algebra $C^*(Y)$
is isomorphic to the coarse $K$-homology $KX_*(Y)$ of $Y$.
As mentioned above, it is proved in \cite{FO-CCH} that 
the coarse Baum-Connes conjecture holds for proper coarsely convex spaces.
It is also proved that the coarse $K$-homology of a coarsely convex space 
$Y$ is isomorphic to the reduced $K$-homology of the ideal boundary 
$\partial Y$ with sifting the degree by one \cite[Theorem 6.7.]{FO-CCH}. Therefore we can compute the $K$-theory of the Roe algebra of a free product coarsely convex spaces as follows.


For a $C^*$-algebra $A$, we denote by $K_p(A)$ the operator $K$-theory of $A$. 
We denote by $\tilde{K}_p(-)$ the reduced $K$-homology.
\begin{theorem}
\label{thm:K-Roe-alg-bdry}
 Let $X$ and $Y$ be proper geodesic coarsely convex spaces with nets. 
 Suppose that the net of $X$ and that of $Y$ are coarsely dense
 in $X$ and $Y$, respectively, and, both $X$ and $Y$ are unbounded.
 We have
 \begin{align*}
  K_p(C^*(X * Y)) 
 &\cong \tilde{K}_{p-1}(\mathcal{C}) \times
  \prod_{\N} \tilde{K}_{p-1}(\partial X)
  \times
  \prod_{\N} \tilde{K}_{p-1}(\partial Y).
\end{align*} 
\end{theorem}

\begin{corollary}
 \label{cor:K-Roe-alg}
 Let $X$ and $Y$ be proper geodesic coarsely convex spaces with nets
 satisfying the condition of \cref{thm:K-Roe-alg-bdry}.
 We have
 \begin{align*}
  K_p(C^*(X * Y)) 
 &\cong \tilde{K}_{p-1}(\mathcal{C}) \times
  \prod_{\N} K_{p}(C^*(X))
  \times
  \prod_{\N} K_{p}(C^*(Y)).
\end{align*} 
\end{corollary}
It might be interesting to compare \cref{cor:K-Roe-alg} with a formula for 
$K$-theories of the group $C^*$ algebras of free products of groups 
\cite[10.11.11 (f) to (h)]{MR1656031}.

We also study the topological dimensions of the ideal boundaries
of trees of spaces. Applying \cref{thm:cd-freeprod-grp} for free products
of metric spaces, we obtain a formula of the dimension.

\begin{theorem}
\label{thm:dim_free-pd}
 Let $X$ and $Y$ be geodesic coarsely convex spaces with nets. 
 \begin{align*}
  \dim (\partial (X*Y)) = \max\{\dim \partial X, \dim \partial Y\}.
 \end{align*}
\end{theorem}

We remark that \cref{thm:dim_free-pd} is an analogue of 
the formula for the cohomological dimensions of the 
free products of groups.
Namely, in the setting of \cref{thm:dim_free-pd},
if we assume that both $X$ and $Y$ admits geometric action by
group $G$ and $H$, respectively, and both $G$ and $H$ admits
a finite model for the classifying space $BG$ and $BH$, respectively,
then the above formula follows from a well-known formula
on the cohomological dimensions of free products of groups.
See \cref{sec:relat-betw-cohom} for details.

\subsection{Outline}
In \cref{sec:comb-metr-horob}, based on the work of Groves and Manning,
we introduce combinatorial horoballs for metric spaces with lattices, and study 
the shape of geodesics in them.

In \cref{sec:coars-conv-bicomb,sec:ideal-boundary,sec:geod-coars-conv}, 
we review the coarsely convex bicombings and the construction of ideal boundaries.
In this paper, we only deal with geodesic coarsely convex space, so
in \cref{sec:geod-coars-conv}, we restrict results in previous sections 
to geodesic coarsely convex bicombings, and we improve some of them.

In \cref{sec:cont-at-infin}, we prepare some technical lemmata to show 
the continuity of certain retractions on ideal boundaries, 
which play key roles in the computations of (co)homologies.

In \cref{sec:trees-spaces,sec:augmented-space}, we introduce trees of spaces and 
augmented spaces. We construct geodesic coarsely convex bicombings on them assuming 
that each component admits such bicombings.

In \cref{sec:cohomology-ideal-boundary-tree}, we compute (co)homologies of ideal boundaries
of trees of geodesic coarsely convex spaces by using Mayer-Vietoris exact sequences.
Key tools are retractions on ideal boundaries.
We give proofs of \cref{thm:cohomology_free_prod,thm:K-Roe-alg-bdry,cor:K-Roe-alg-tree}. 

In \cref{sec:constr-retarct}, we construct the retractions 
used in \cref{sec:cohomology-ideal-boundary-tree} and show that they are continuous.
Here we use the results in \cref{sec:cont-at-infin}.

In \cref{sec:ideal-bound-augm}, we study the topological type of the ideal boundary
of the augmented spaces. 
We show (\cref{thm:Cantor}) that under certain conditions, the ideal boundaries of
the augmented spaces are homeomorphic to the Cantor space.

In \cref{sec:topol-dimens-ideal}, we show a formula for the topological dimensions of 
the ideal boundaries of trees of geodesic coarsely convex spaces in terms of 
the dimensions of each component. In \cref{sec:relat-betw-cohom}, we compare these
results with a well known formula for the cohomological dimensions of the free products
of groups.

\section{Combinatorial and metric horoballs}
\label{sec:comb-metr-horob}
All materials in this section are slight modifications of the work of Groves and Manning
\cite{MR2448064} on the combinatorial horoballs.

In the rest of this section, we assume that $X$ is a geodesic metric
space.

\begin{definition}
 Let $(X,d_X)$ be a metric space and $A\subset X$. Let $C\geq 0$.
 \begin{enumerate}[label=(\arabic*)]
  \item The subset $A$ is $C$-discrete if for all $a,a'\in A$ with $a\neq a'$, we have
        $d_Z(a,a')\geq C$.
  \item The subset $A$ is $C$-dense if for all $x\in X$, there exists $a\in A$ such that
        $d_Z(x,a)\leq C$.
 \end{enumerate}
\end{definition}

\begin{remark}
 A maximal $C$-discrete subset $A\subset X$ is $2C$-dense.
\end{remark}

\begin{definition}
\label{def:lattice}
 A metric space with a lattice is a pair $(X,X^{(0)})$ of metric space $(X,d_X)$ and
 1-discrete 2-dense subset $X^{(0)}\subset X$. 
 We call $X^{(0)}$ \emph{the lattice of} $X$.
\end{definition}

Groves and Manning defined an augmented space and introduced an equivalent definition of relatively hyperbolic groups \cite{MR2448064}.
The augmented space is constructed by gluing a ''combinatorial horoballs'' based on each $P \in \mathcal{P}$ onto the Cayley graph of $G$. 
They defined that $G$ is hyperbolic relative to $\mathcal{P}$ if the augmented space is hyperbolic. 
\par
The construction presented here is closely modeled on Groves–Manning’s combinatorial horoballs. 
They constructed horoballs based on an arbitrary connected graph, 
and showed that such a horoball is always $\delta$–hyperbolic. 
A minor modification of the Groves–Manning construction produces connected horoballs based on any metric space. 

\begin{definition}
\label{def:conb-horob}
Let $(X^{(0)},d_X)$ be a 1-discrete metric space.
A \textit{combinatorial horoball} $\cbHoro(X^{(0)})$ is a graph
with a set of vertices $\cbHoro(X^{(0)})^{(0)}$ and a set of edges
$\cbHoro(X^{(0)})^{(1)}$ defined as follows:
\begin{enumerate}
    \item $\cbHoro(X^{(0)})^{(0)} \coloneqq X^{(0)} \times (\mathbb{N} \cup \{0\})$.
    \item $\cbHoro(X^{(0)})^{(1)}$ contains the following two types of edges:
    \begin{enumerate}
        \item For each $x \in X^{(0)}$ and $l \in \mathbb{N} \cup \{0\}$, 
              there exists a \textit{vertical edge} connecting $(x,l)$ and $(x,l+1)$.
        \item For each $x,y \in X^{(0)}$ and $l \in \mathbb{N}$, 
              if $0 < d_X(x,y) \leq 2^l$ holds for some $l \in \mathbb{N}$, 
              there exists a \textit{horizontal edge} connecting $(x,l)$ and $(y,l)$.
    \end{enumerate}
\end{enumerate}
We endow $(X^{(0)},d_X)$ with a graph metric such that
each edge has length one, and we consider $(X^{(0)},d_X)$ 
as a geodesic space.
\end{definition}

Let $\cbHoro(X^{(0)})$ be a combinatorial horoball.
A \textit{horizontal segment} is a path consisting only of horizontal edges.
A \textit{vertical segment} is a path consisting only of vertical edges. 
Each $x \in \cbHoro(X^{(0)})^{(0)}$ can be identified with $x=(x_0,l)$, where $x_0 \in X^{(0)}$ and $l \in \mathbb{N} \cup \{0\}$. 
We say that the depth of $x$ is $l$, denoted by $D(x)=l$. 
Moreover, let $p$ be a horizontal segment.
When the depth of a point on the horizontal segment $p$ is $l$,
we say that the depth of the horizontal segment $p$ is $l$, denoted by $D(p)=l$.


\begin{definition}
\label{def:metr-horob}
 Let $(X, X^{(0)})$ be a metric space with a lattice.
An \textit{metric horoball} $\Horo(X,X^{(0)})$ is a
quotient space of $X \sqcup \cbHoro(X^{(0)})$ by the
 equivalent relation generated by 
\begin{align*}
  \cbHoro(X^{(0)})^{(0)}\ni (x_0,0) \sim x_0\in X^{(0)}\subset X 
\end{align*} 
for any $x_0 \in X^{(0)}$.
\end{definition}

Geodesics in combinatorial horoballs are particularly easy to understand.

\begin{proposition}[\cite{MR2448064}, Lemma 3.10]
\label{prop:geod}
Let $x,y \in \Horo(X,X^{(0)})$.  
There exists a geodesic segment on $\Horo(X,X^{(0)})$ from $x$ to $y$ which consists of at most two vertical
segments and a single horizontal segment of length at most $5$.
We say that this geodesic segment is a \textit{normal geodesic segment}, denoted by $\ngamma(x,y)$. 
Moreover, for each geodesic segment from $x$ to $y$, denoted by $\gamma(x,y)$, we have
\begin{align*}
    d_H(\mathrm{Im}(\gamma(x,y)), \mathrm{Im}(\ngamma(x,y)))\leq 4
\end{align*}
,where $d_H$ is the Hausdorff metric on $X$.
\end{proposition}

\begin{proof}
Let $\gamma(x,y)$ be a geodesic segment from $x$ to $y$. 
We can put 
\begin{align*}
    \gamma(x,y)=p_1 \ast p_2 \ast \cdots \ast p_n
\end{align*}
such that for all $i$, we have that
if $p_i$ is a horizontal segment,
then $p_{i+1}$ is a vertical segment,
and if $p_i$ is a vertical segment, 
then $p_{i+1}$ is a horizontal segment.
We can assume that $p_1$ and $p_n$ are vertical segments.
Let $s(p_i)$ be an initial point of $p_i$ and let $t(p_1)$ be a terminal point of $p_i$. 
Note that $t(p_i)=s(p_{i+1})$, $s(p_1)=x$, and $t(p_n)=y$. 
\par
First, it is easily shown that there does not exist a vertical segment $p$ such that $D(s(p)) > D(t(p))$. 
We will show that as a consequence of lemmas.

\begin{claim}\label{cla:form1}
Let $D_M$ be the maximum depth of horizontal segments which compose $\gamma(x,y)$.
If there exists a horizontal segment $p$ which composes $\gamma(x,y)$ such that $D(p) < D_M$, 
then the length of $p$ is $1$. 
\end{claim}

Let $p_i$ be a horizontal segment with $D(p_i) < D_M$. 
Suppose the length of $p_i$ is $2$.
Let $s(p_i)=(x_i, D(p_i))$ and $t(p_i)=(y_i, D(p_i))$.
Note that $p_{i+1}$ is a vertical segment and $t(p_i)=s(p_{i+1})$. 
We denote the depth of $t(p_{i+1})$ by $D_1$.
Since $D(p_i) < D_1$, 
there exists a horizontal edge from $(x_i, D_1)$ to $(y_i,D_1)$, denoted by $q_2$.
Let $q_1$ be the vertical segment from $s(p_i)=(x_i, D(p_i))$ to $(x_i,D_1)$.
The path $q_1 \ast q_2$ is a geodesic segment from $s(p_i)$ to $t(p_{i+1})$. 
The length of $q_1 \ast q_2$ is shorter than $p_i \ast p_{i+1}$.
This is a contradiction. 

\begin{claim}\label{cla:form2}
There does not exist a subgeodesic $h_1 \ast v_1 \ast h_2 \ast v_2$ such that $h_i$ are horizontal segments and $v_i$ are vertical segments. 
\end{claim}

Suppose that there exists a subgeodesic $h_1 \ast v_1 \ast h_2 \ast v_2$ such that $h_i$ are horizontal segments and $v_i$ are vertical segments. 
Let $D(h_1)=D_1$, $D(h_2)=D_2$, and $D(t(v_2))=D$. 
By \cref{cla:form1}, the length of $h_1$ and $h_2$ are $1$. 
We can put $s(h_1)=(x_1,D_1)$, $s(h_2)=(x_2,D_2)$, and $t(v_2)=(x_3,D)$
\begin{align*}
    \ds{x_1,x_3} 
    &\leq \ds{x_1,x_2}+ \ds{x_2,x_3} \\
    &\leq 2^{D_1} + 2^{D_2} \\
    &\leq 2^{D_2+1} \leq 2^{D}. 
\end{align*}
Therefore, there exits a horizontal edge $q_2$ from $(x_1,D)$ to $t(v_2)$.
Now, we define a path $q_1$ to be a vertical segment from $s(h_1)=(x_1, D_1)$ to $(x_1, D)$. 
The length of the path $q_1 \ast q_2$ is smaller than $h_1 \ast v_1 \ast h_2 \ast v_2$.
This is a contradiction. 



Therefore, a geodesic segment $\gamma(x,y)$ has the following form:
\begin{align*}
    \gamma(x,y)=v_1 \ast h_1 \ast v_2 \ast h_2 \ast v_3 \ast h_3 \ast v_4,
\end{align*}
where $v_i$ are vertical segments and $h_i$ are horizontal segments. 
Note that $h_2$ is the highest horizontal segment. 
By \cref{cla:form1} and \cref{cla:form2}, we have that the length of $h_i$ is $1$ for $i \in \{1,3\}$. 
Put $s(h_1)=(x,D(h_1))$ and $t(h_3)=(y,D(h_3))$. 
We denote by $v'_1$, a vertical segment from $(x,D(h_1))$ to $(x,D_M)$.
and denote by $v'_2$, a vertical segment from $(y,D_M)$ and $t(h_3)=(y,D(h_3))$. 
There exists a geodesic segment from $(x,D_M)$ to $(y,D_M)$, denoted by $h_M$.
We replace $h_1 \ast v_2 \ast h_2 \ast v_3 \ast h_3$ with $v'_1 \ast h_M \ast v'_2$. 
The replaced path is also a geodesic segment from $x$ to $y$. 
Therefore, we obtain a normal geodesic segment from $x$ to $y$. 
\par
For $x,y \in \Horo(X,X^{(0)})$, 
let $\ngamma(x,y)=u \ast h \ast v$, where $u$ and $v$ are vertical segments 
and $h$ is a horizontal segment. 
Finally, we will show that the length of $h$ is at most $5$. 
Suppose that the length of $h$ is $6$.
Let $s(h)=(p,D(h))$ and $t(h)=(q,D(h))$. 
We can connect $(p,D(h)+1)$ and $(q,D(h)+1)$ with three horizontal edges.
Then, there exists a path of length $5$ connecting $s(h)$ and $t(h)$.
This is the contradiction. 
\par 
Moreover, for each geodesic segment from $x$ to $y$, denoted by $\gamma(x,y)$, we have
\begin{align*}
    d_H(\mathrm{Im}(\gamma(x,y)), \mathrm{Im}(\ngamma(x,y)))\leq 4. 
\end{align*}
This completes the proof.  
\end{proof}

\begin{proposition}[\cite{MR2448064}]
\label{prop:Horoball-hyp}
$\Horo(X,X^{(0)})$ is a Gromov hyperbolic space.
\end{proposition}

\begin{proof}

Let $x,y,z \in X^{(0)}$.
We denote by $[x,y]$, an image of a geodesic segment on $\Horo(X,X^{(0)})$ from $x$ to $y$.
Let $p_1 \in [x,y]$, $p_2 \in [y,z]$, and $p_3 \in [z,x]$. 
We define
\begin{align*}
    \diam\{p_1,p_2,p_3\} \coloneqq 
    \max\{d_X(p_i,p_j) \colon i \neq j\}.
\end{align*}
One of the equivalent definitions of Gromov hyperbolicity is that the following quantity 
\begin{align*}
    \min \{ \diam \{p_1,p_2,p_3\} \colon p_1 \in [x,y], p_2 \in [y,z], p_3 \in [z,x] \}
\end{align*}
is bounded by a constant independent of choices of geodesic segments and geodesic triangles. See \cite[21.- Proposition]{MR1086648} for details.
We say that this quantity is the minimum diameter of geodesic triangle $[x,y] \cup [y,z] \cup [z,x]$. 
\par
By \cref{prop:geod}, there exists a normal geodesic segment connecting each two points, that is,
\begin{itemize}
\item 
$\ngamma(x,y)=u_1 \ast h_1 \ast v_1$,
\item 
$\ngamma(y,z)=u_2 \ast h_2 \ast v_2$,
\item 
$\ngamma(z,x)=u_3 \ast h_3 \ast v_3$,
\end{itemize}
where $u_i$ and $v_i$ are vertical segments
and each $h_i$ is a horizontal segment. 
Suppose $D(h_1) \leq D(h_2) \leq D(h_3)$. 
\par
It is easily shown that $\Im(u_1) \subset \Im(v_3)$ and $\Im(v_1) \subset \Im(u_2)$. 
Then, we have
\begin{align*}
    \Im(h_1) 
    \subset N(\Im(v_3),3) \cup N(\Im(u_2),3). 
\end{align*}
Therefore, 
\begin{align*}
    \min \{ \diam \{p_1,p_2,p_3\} \colon p_1 \in \Im(\ngamma(x,y)), p_2 \in \Im(\ngamma(y,z)), p_3 \in \Im(\ngamma(z,x)) \}
\end{align*}
is bounded by $5$. 
By \cref{prop:geod},
for each geodesic segment from $x$ to $y$, denoted by $\gamma(x,y)$, we have
\begin{align*}
    d_H(\mathrm{Im}(\gamma(x,y)), \mathrm{Im}(\ngamma(x,y)))\leq 4. 
\end{align*}
For any $x,y,z \in \Horo(X,X^{(0)})$ and any geodesic triangle $\Delta(x,y,z)$, 
the minimal diameter is bounded by $9$. 
Therefore, $\Horo(X,X^{(0)})$ is a Gromov hyperbolic space. 
%
%
\end{proof}

\section{coarsely convex bicombing}
\label{sec:coars-conv-bicomb}
Let $(X,d_X)$ be a metric space.
For $\lambda \geq1$ and $k \geq 0$, 
a $(\lambda,k)$-\textit{quasi-geodesic bicombing} on $X$ is 
a map $\gamma \colon X \times X \times [0,1] \to X$ 
such that for $x,y \in X$, we have $\gamma(x,y,0)=x$, $\gamma(x,y,1)=y$, and 
\begin{align*} 
\lambda^{-1}\abs{t-s}\ds{x,y} - k \leq \ds{\gamma(x,y,t), \gamma(x,y,s)} 
 \leq \lambda \abs{t-s}\ds{x,y} +k \ \ \ \ (t,s \in [0,1]). 
\end{align*}

If the space  $X$ admits a $(\lambda,k)$-quasi-geodesic bicombing for some $\lambda \geq 1$ and $k>0$,
we say that $X$ admits a quasi-geodesic bicombing. 
If a group $G$ acts on $X$ by isometries, 
a geodesic bicombing $\gamma : X \times X \times [0,1] \to X$ is
\textit{$G$-equivariant} if 
\begin{align*}
g \cdot \gamma(x,y)(t)=\gamma(gx,gy)(t)
\end{align*}
holds for any $g \in G$,   $x,y \in X$, and  $t \in [0,1]$.

\begin{definition}
\label{def:ccbicombing}
Let $\lambda\geq 1$, $k \geq 0$, $E\geq 1$,
and $C\geq 0$ be constants. 
Let $\theta\colon \Rp\to\Rp$ be a non-decreasing function.

A $(\lambda,k,E,C,\theta)$
-\textit{coarsely convex bicombing} 
on a metric space $(X,d_X)$ is a $(\lambda,k)$-quasi-geodesic bicombing 
$\gamma \colon X \times X \times [0,1] \to X$ with the following:
 \begin{enumerate}[label=(\roman*)]
 \item \label{qconvex}
       Let $x_1,x_2,y_1,y_2 \in X$ and let $a,b \in [0,1]$.
       Set $y_1':=\gamma(x_1,y_1,a)$ and $y_2':=\gamma(x_2,y_2,b)$.
       Then, for $c \in [0,1]$, we have
 \begin{align*} 
  \ds{\gamma(x_1,y_1,ca), \gamma(x_2,y_2,cb)} 
  \leq (1-c)E \ds{x_1,x_2} + cE \ds{y_1',y_2'} + C.
 \end{align*}
 \item \label{qparam-reg}       
       Let $x_1,x_2,y_1,y_2 \in X$.
       Then for $t,s\in [0,1]$ 
       we have
       \begin{align*}
	\abs{t\ds{x_1,y_1} - s\ds{x_2,y_2}} \leq 
        \theta(\ds{x_1,x_2}+\ds{\gamma(x_1,y_1,t),\gamma(x_2,y_2,s)}).
       \end{align*}
 \end{enumerate} 
\end{definition}

The following reparametrization is used to construct ideal boundaries
 in \cref{sec:ideal-boundary}
\begin{definition}
 Let $X$ be a metric space and 
 let $\gamma\colon X\times X \times [0,1]\to X$ be a 
 $(\lambda,k)$-quasi-geodesic bicombing on $X$.
 A reparametrised bicombing of $\gamma$ is a map 
 \begin{align*}
  \rp{\gamma}\colon X\times X\times \R_{\geq 0} \to X 
 \end{align*}
 defined by
 \begin{align*}
  \rp{\gamma}(x,y,t):= 
  \begin{cases}
   \gamma\left(x,y,{t}/{\ds{x,y}}\right) & \text{if } t\leq \ds{x,y}\\
   y& \text{if } t> \ds{x,y}
  \end{cases}.
 \end{align*}
\end{definition}

\section{Ideal boundary}
\label{sec:ideal-boundary}
In this section, we review the construction of the ideal boundary
of a coarsely convex space.

Let $(X,d_X)$ be a metric space equipped with a 
$(\lambda,k,E,C,\theta)$-coarsely convex bicombing 
$\gamma\colon X\times X \times [0,1]\to X$. 
We choose a base point $e\in X$.

\subsection{Gromov product and sequential boundary}
\label{sec:sq-boundary}
\begin{definition}
\label{def:Gromov-prod}
Set $k_1=\lambda + k$, $D= 2(1+E)k_1+ C$, and $D_1 = 2D+2$.
 We define a product $(\cdot \mid \cdot)_e \colon X\times X \to \Rp$ by 
 \begin{align*}
  (x\mid y)_e:= \min\{\ds{e,x},\ \ds{e,y},\ 
  \sup\{t\in \Rp:\ds{\rp{\gamma(e,x,t)},\rp{\gamma(e,y,t)}}\leq D_1\}\},
 \end{align*}
We abbreviate $(x\mid y)_e$ by $(x\mid y)$.
\end{definition}

\begin{lemma}[{\cite[Lemma 4.8]{FO-CCH}}]
\label{lem:qultm}
Set $D_2:=E(D_1+2k_1)$.
 For $x,y,z\in X$, we have
 \begin{align*}
  (x\mid z) \geq D_2^{-1}\min
  \{(x\mid y), (y\mid z)\}.
 \end{align*}
\end{lemma}

We use a sequential model of the ideal boundary introduced in 
\cite{FOYcont-prod}. The idea is based on \cite{MR919829} 
and \cite[Chapitre 7]{MR1086648}. 

\begin{definition}
\label{def:Sinfty}
 Let 
 \begin{align*}
  S_\infty (X) =\{ (x_i) : (x_i\mid x_j) \to \infty 
  \text{ as } i,j \to \infty\}
 \end{align*}

 and define a relation $\sim $ on $S_\infty (X)$ as follows.
 For every $(x_i), (y_i) \in S_{\infty} (X)$, we have
 $(x_i) \sim  (y_i)$ if
 \begin{align*}
 (x_i\mid y_i) \to \infty \text{ as } i \to \infty
 \end{align*}
\end{definition}

\begin{lemma}
\label{fact:S-infty}
The relation $\sim $ is an equivalence relation on $S_\infty (X)$.
\end{lemma}
\begin{proof}
It is clear that the relation $\sim$ is 
reflexive and symmetric. \cref{lem:qultm} implies that 
it is transitive.
\end{proof} 

\begin{definition}
\label{def:gromov-boundary}
Let 
\begin{align*}
 \partial X = S_\infty (X) / \sim \quad \text{ and } \quad \overline{X} = X \cup \partial X.
\end{align*}

For $x \in X$ and a sequence $(x_i)$ in $X$,
we write $(x_i) \in x$ if $x_i=x$ for every $i \in \mathbb{N}$.
We extend the Gromov product 
$(\cdot \mid \cdot)\colon X\times X \to \Rp $ to a symmetric function 
$(\cdot \mid \cdot)\colon \overline{X}\times \overline{X} 
\to \Rp \cup \{\infty\} $ by letting 
\begin{align}
\label{eq:ext-cp}
(x \mid y) 
&=\sup \{\liminf_{i \to \infty} (x_i\mid y_i) : (x_i) \in x , (y_i)\in y   \}
\end{align}
for  $x, y \in \overline{X}$.
\end{definition}

For $n\in \mathbb{N}$, let 
\begin{align*}
 V_n &=\{ (x,y)\in \overline{X} \times \overline{X} : (x\mid y ) >n\} \cup \{ (x,y) \in X\times X : \ds{x,y}<1/n\},\\
 V_n[x] &= \{y\in \overline{X}: (x,y)\in V_n\}, \quad (x\in \overline{X}).
\end{align*}

By \cite[Lemma 2.6.]{FOYcont-prod}, 
$\{V_n : n\in \mathbb{N}\}$ is a base of 
a metrizable uniformity on $\overline{X}$. 
See also~\cite[Section 4.3]{FO-CCH}.

\begin{definition}
\label{def:gromov-bdry}
 Let $\overline{X}$ be equipped with the topology 
 $\mathcal{T}_{(\cdot\mid \cdot)}$
 generated by 
 the family $\{V_n[x] : x \in \overline{X}, n \in \mathbb{N}\}$.
 We call the subspace $\partial X$ of $\overline{X}$  
 the \emph{Gromov boundary} of $X$ with respect to $(\cdot \mid \cdot)$.
 We also call $\partial X$ the \emph{ideal boundary} of $X$.
\end{definition}

\begin{lemma}[{\cite[Lemma 2.8, Theorem 2.9]{FOYcont-prod}}]
\label{lem:Gromov-cptn}
We have the following:
 \begin{enumerate}[label=(\roman*)]
  \item 
        \label{fact:item:cptn-top}
        The relative topology $\mathcal{T}_{(\cdot\mid \cdot)}\rest_X$ 
        on $X$ with respect 
        to $\overline{X}$
        coincides with the topology $\mathcal{T}_d$ 
        induced by the metric $d$. 
  \item 
      $X$ is a dense open subset in $\overline{X}$.

  \item 
        If $(X,d)$ is complete, then so is $(\overline{X}, d)$. 
  \item If $(X,d)$ is proper, then $(\overline{X}, d)$ is compact. 
 \end{enumerate}
\end{lemma}

\subsection{Combing at infinity}




\begin{lemma}
\label{lem:extended_bicombing}
 Suppose that $X$ is proper.
 Then there exists a map 
 \begin{align*}
  \rpgm\colon X\times \overline{X}\times \Rp\to X
 \end{align*} 
 satisfying the following:
 \begin{enumerate}[label=(\roman*)]
  \item \label{item:extension}
        For $x,y\in X$, we have 
        $\rpgm(x,y,-) = \rp{\gamma(x,y,-)}$.

  \item  \label{item:approximation}
         For $(e,x)\in X\times \partial X$, there exists a 
         sequence $(x_n)$ in $X$
         such that the sequence of maps 
         $(\rp{\gamma}(e,x_n,-)\rest_\N:\N\to X)_{n}$ 
         converges to $\rpgm(e,x,-)\rest_\N$ pointwise.
 
 \item \label{item:rpgm-lim-gprod}
        For $(e,x)\in X\times \partial X$, we have
        \begin{align*}
         (\rpgm(e,x,t)\mid x)_e \to \infty \quad \text{as} \quad t\to \infty
        \end{align*}

  \item \label{item:qasi-geod-bicombing-at-infty}
        For $(e,x)\in X\times \partial X$, the map
        $\rpgm(e,x,-)\colon \Rp\to X$ is a 
        $(\lambda,k_1)$-quasi geodesic, where $k_1:=\lambda + k$.
 \end{enumerate}
\end{lemma}

\begin{proof}
 For $(x,y)\in X\times X$, we define $\rp{\gamma(x,y,-)}$ by
 the equality in \ref{item:extension}.
 We use~\cite[Proposition 4.17]{FO-CCH} to construct 
 maps $\rpgm(e,x,-)$ for $(e,x)\in X\times \partial X$ satisfying
 \ref{item:approximation} and \ref{item:rpgm-lim-gprod}.
 Finally, \ref{item:qasi-geod-bicombing-at-infty} follows 
 from~\cite[Lemma 4.1]{FO-CCH}. 
\end{proof}

\begin{definition}
 We call the map $\rpgm$ given in \cref{lem:extended_bicombing} 
 an \emph{extended bicombing} on 
 $X\times \overline{X}$ corresponding to $\gamma$.
 For $(e,x)\in X\times \partial X$, we abbreviate 
 $\rpgm(e,x,-)$ by $\gamma_e^x(-)$.
\end{definition}

\begin{lemma}[{\cite[Proposition 4.2]{FO-CCH}}]
\label{lem:rpgm_cconvex}
 The extended bicombing $\rpgm$ on $X\times \overline{X}$ satisfies
 the following:
 \begin{enumerate}[label=(\arabic*)]
  \item \label{rpgm-qconvex}
       Let $(e_1,x_1),(e_2,x_2)\in X\times \partial X$ 
        and let $a,b \in \Rp$.
       Set $x_1':=\rpgm(e_1,x_1,a)$ and $x_2':=\rpgm(e_2,x_2,b)$.
       Then, for $c \in [0,1]$, we have
        \begin{align*} 
         \ds{\rpgm(e_1,x_1,ca), \rpgm(e_2,x_2,cb)}
         \leq (1-c)E\ds{e_1,e_2} + cE\ds{x_1',x_2'} + D
        \end{align*}
       	where $D:= 2(1+E)k_1+ C$.
 \item \label{rpgm-qparam-reg}       
       We define a non-decreasing function
       $\tilde{\theta}\colon \Rp\to \Rp$ by
       $\tilde{\theta}(t):=\theta(t+1)+1$.
       Let $(e_1,x_1),(e_2,x_2)\in X\times \partial X$ 
       Then for $t,s\in \Rp$, we have
       \begin{align*}
	\abs{t - s} \leq 
        \tilde{\theta}(\ds{e_1,e_2}+\ds{\rpgm(e_1,x_1,t),\rpgm(e_2,x_2,s)}).
       \end{align*}
 \end{enumerate}
\end{lemma}

\begin{definition}
  For $x,y\in \partial X$, we define
 \begin{align*}
  (\gamma_e^x\mid \gamma_e^y)_e 
  :=\sup\{t\in \Rp:\ds{\gamma_e^x(t),\gamma_e^y(t)}\leq D_1\}.
 \end{align*}
 For $x\in \partial X$ and for $p\in X$, we define
 \begin{align*}
  (\gamma_e^x\mid p)_e = (p\mid \gamma_e^x)_e
  :=\min\{\ds{e,p},\sup\{t\in \Rp:\ds{\gamma(e,p,t),\gamma_e^x(t)}\leq D_1\}\}.
 \end{align*} 
\end{definition}

\begin{lemma}
\label{lem:univ-const}
 There exists a constant $\DA\geq 1$ depending on 
 $\lambda,k,E,C,\theta(0)$ such that
 the following holds:
 \begin{enumerate}[label=(\arabic*)]
  \item \label{lem:D3}
        For $x,y\in \partial X$, we have
   \begin{align*}
   (\gamma_e ^x\mid \gamma_e^y) 
    \leq  (x\mid y) \leq \DA(\gamma_e ^x\mid \gamma_e^y).
  \end{align*}
  \item \label{lem:qultmD}
 For triplet $x,y,z\in \partial X$, we have
 \begin{align*}
  (\gamma_e^x\mid \gamma_e^z) 
    \geq \DA^{-1}
  \min\{(\gamma_e ^x\mid \gamma_e^y),
  (\gamma_e ^y\mid \gamma_e^z)\}.
 \end{align*}
  \item \label{lem:qrayultmvisu} 	
	For triplet $x,y,z\in \overline{X}$, 
 we have 
 \begin{align*}
  (x\mid z) \geq \DA^{-1}\min \{(x\mid y), (y\mid z)\}.
 \end{align*}
  \item \label{lem:maximizer}
 Let $x,y\in \partial X$. For all $t\in \Rp$ with 
	$t\leq (\gamma_e^x\mid \gamma_e^y)$, we have
 \begin{align*}
  \ds{\gamma_e^x(t), \gamma_e^y(t)} \leq \DA.
 \end{align*}
  \item \label{lem:ray-same-param}
        Let $e\in X$ and let $x,y\in \overline{X}$.
 If $\rpgm(e,x,a) = \rpgm(e,y,b)$ 
        for some $a,b\in \Rp$,
	then for all $t\in [0,\max\{a,b\}]$
	we have
 \begin{align*}
  \ds{\rpgm(e,x,t),\rpgm(e,y,t)}\leq \DA.
 \end{align*} 
  \item \label{gprod-contraction}
        Let $e\in X$ and $x\in \partial X$. For $v\in X$ and $t\in [0,1]$,
        we have
        \begin{align*}
         (x\mid \gamma(e,v,t))_e \geq 
         \Omega^{-1} \min\{(x\mid v)_e, td_X(e,v)\}.
        \end{align*}
 \end{enumerate}
\end{lemma}

\begin{proof}
 \ref{lem:D3} to \ref{lem:ray-same-param} are 
 \cite[Lemma 2.9.]{EzawaFukayaKobeJ}. We give a proof of \ref{gprod-contraction}.
 By the definition, we have $(x\mid \gamma(e,v,t))_e \geq td_X(e,v)$.
 We suppose that $(x\mid v)\leq td_X(e,v)$. Set $s=(x\mid v)$. By \cite[lemma 4.7.]{FO-CCH},
 we have $d_X(\gamma_e^x(s),\rp\gamma(e,v,s))\leq D_1+k_1$. 
 Set $s'\coloneqq s/(E(D_1 + k_1))$. Then,
 \begin{align*}
  d_X(\gamma_e^x(s'),\rp\gamma(e,v,s'))\leq \frac{E}{E(D_1 + k_1)}
  d_X(\gamma_e^x(s),\rp\gamma(e,v,s)) + D\leq D+1\leq D_1
 \end{align*}
 Thus $(x\mid \gamma(e,v,t))\geq (E(D_1 + k_1))^{-1}(x\mid v)$.

\end{proof}

\subsection{Visual maps}
Let $(X,d_X)$ and $(Y,d_Y)$ be metric spaces.
We say that a map $f\colon X\to Y$ is a \emph{large scale Lipschitz} map, if
there exists $L>1$ such that 
\begin{align*}
 d_Y(f(x),f(y)) \leq L d_X(x,y) + L.
\end{align*}
Let $\gamma_x$ and $\gamma_y$ be coarsely convex bicombings on $X$ and $Y$, respectively.
Let $a\in X$ and $b\in Y$ be base points of $X$ and $Y$, respectively.

\begin{definition}
\label{def:visual-map}
 We say that a large scale Lipschitz map $f\colon X\to Y$ is \emph{visual} if for 
 every pair of sequences $(x_n)_{n}, (y_n)_{n}$ in $X$ with 
 $(x_n\mid y_n)_{a} \to \infty$, we have $(f(x_n)\mid f(y_n))_b \to \infty$.
\end{definition}

Let $f\colon X\to Y$ be a visual map. We define a map
$\partial f\colon \partial X \to \partial Y$ as follows. 
For $x \in \partial X$, we choose a sequence $(x_n)_n$ representing $x$.
Then we define $\partial f(x)$ to be the equivalence class of 
the sequence $(f(x_n))_n$. By \cite[Proposition 3.5]{EzawaFukayaKobeJ}, 
$\partial f(x)$ does not depend on the choice of $(x_n)_n$.
We say that $\partial f$ is \textit{induced} by $f$.

\begin{proposition}[{\cite[Corollary 3.6]{EzawaFukayaKobeJ}}]
\label{prop:visual-df-conti}
 Let $f$ be a visual map. Then the induced map 
 $\partial f\colon \partial X\to \partial Y$ is continuous.
\end{proposition}

\section{Geodesic coarsely convex spaces}
\label{sec:geod-coars-conv}
\subsection{Geodesic combing at infinity}

Let $(X,d_X)$ be a metric space.
A \textit{geodesic bicombing} on $X$ is a map 
$\gamma : X \times X \times [0,1] \to X$ 
such that for $x,y \in X$, we have
$\gamma(x,y,0)=x$, $\gamma(x,y,1)=y$, and 
\begin{align*}
 \ds{\gamma(x,y,t), \gamma(x,y,s)} = \abs{t-s} \ds{x,y} \quad
 (t,s \in [0,1])
\end{align*}

\begin{definition}
\label{def:g-ccbicombing}
Let $E\geq 1$ and $C\geq 0$ be constants. 
An $(E,C)$-\textit{geodesic coarsely convex bicombing} 
on a metric space $(X,d_X)$ is a geodesic bicombing 
$\gamma \colon X \times X \times [0,1] \to X$ satisfying the following:

 Let $x_1,x_2,y_1,y_2 \in X$ and let $a,b \in [0,1]$.
 Set $y_1':=\gamma(x_1,y_1,a)$ and $y_2':=\gamma(x_2,y_2,b)$.
 Then, for $c \in [0,1]$, we have
 \begin{align*} 
  d_X(\gamma(x_1,y_1,ca), \gamma(x_2,y_2,cb)) 
  \leq (1-c)Ed_X(x_1,x_2) + cEd_X(y_1',y_2') + C.
 \end{align*}
\end{definition}

\begin{lemma}
\label{lem:gccb-is-qgccb}
 An $(E,C)$-\textit{geodesic coarsely convex bicombing} $\gamma$
 on a metric space $(X,d_X)$ is a 
 $(1,0,E,C,\mathrm{id}_{\Rp})$-\textit{coarsely convex bicombing} 
 on $(X,d_X)$.
\end{lemma}

\begin{proof}
 It follows from the definition of the $(E,C)$-\textit{geodesic coarsely convex bicombing},
 that $\gamma$ satisfies \ref{qconvex} of \cref{def:ccbicombing}.

 Let $x_1,x_2,y_1,y_2 \in X$. Then for $t,s\in [0,1]$, we have 
 \begin{align*}
  \abs{t \ds{x_1,y_1} - s \ds{x_2,y_2}} \leq 
  \ds{x_1,x_2}+\ds{\gamma(x_1,y_1,t),\gamma(x_2,y_2,s)}.
 \end{align*}
 by the triangle inequality.
 Therefore $\gamma$ satisfies \ref{qparam-reg} of \cref{def:ccbicombing}
 with $\theta = \textrm{id}_{\Rp}$.
\end{proof}

\begin{example}
\label{eg:hypsp-canbicomb}
 Let $(X,d_X)$ be a geodesic $\delta$-hyperbolic space. For $x,y\in X$,
 we choose a geodesic $\gamma_{x,y}\colon [0,\ds{x,y}] \to X$ such that
 $\gamma_{x,y}(0)=x$ and $\gamma_{x,y}(\ds{x,y})=y$. We define
 $\gamma\colon X\times X\times [0,1] \to X$ by
 \begin{align*}
  \gamma(x,y,t):=\gamma_{x,y}(t/d_X(x,y)).
 \end{align*}
 Then $\gamma$ is a $(1,8\delta)$-geodesic coarsely convex bicombing.
 See \cite[Chapitre 2. 25.- Proposition]{MR1086648} for details. 
\end{example}

Let $X$ be a space equipped with an
$(E,C)$-geodesic coarsely convex bicombing 
$\gamma\colon X\times X \times [0,1]\to X$.

By \cref{lem:gccb-is-qgccb} and \cref{lem:extended_bicombing}, we have an extended bicombing on
$X\times \overline{X}$. Since $\gamma$ is a geodesic bicombing, we 
can upgrade it.

\begin{lemma}
\label{lem:g-extended_bicombing}
 Suppose that $X$ is proper and $\gamma$ is an 
 $(E,C)$-geodesic coarsely convex bicombing on $X$.
 Then there exists a map 
 \begin{align*}
  \rpgm\colon X\times \overline{X}\times \Rp\to X
 \end{align*} 
 satisfying the following:
 \begin{enumerate}[label=(\arabic*)]
  \item \label{item:g-extension}
        For $x,y\in X$, we have 
        $\rpgm(x,y,-) = \rp{\gamma(x,y,-)}$.

  \item  \label{item:g-approximation}
         For each $(e,x)\in X\times \partial X$, there exists a 
         sequence $(x_n)$ in $X$
         such that the sequence of maps 
         $(\rp{\gamma}(e,x_n,-):\Rp\to X)_{n}$ 
         converges to $\rpgm(e,x,-)$ uniformly on compact sets.

 \item \label{item:g-rpgm-lim-gprod}
        For each $(e,x)\in X\times \partial X$, we have
        \begin{align*}
         (\rpgm(e,x,t)\mid x)_e \to \infty \quad \text{as} \quad t\to \infty
        \end{align*}

  \item \label{item:geod-bicombing-at-infty}
        For each $(e,x)\in X\times \partial X$, the map
        $\rpgm(e,x,-)\colon \Rp\to X$ is a geodesic.
 \end{enumerate}
\end{lemma}

\begin{proof}
 We can use the same construction of geodesics as the one described in
 \cite[5-25 Th\'eor\`em]{MR1086648}. 
 Here we only give a sketch of the proof.

 Let $(e,x)\in X\times \partial X$. We choose 
 $(x_n)\in S_\infty(X)$ such that $(x_n)\in x$. 
 By Arzel\`{a}–Ascoli Theorem and standard diagonal arguments,
 We can take a subsequence $x_{n_k}$ such that the sequence of maps
 $(\rp{\gamma}(e,x_{n_k},-):\Rp\to X)_k$ 
 converges uniformly on compact sets. 
 This defines a map $\rpgm(e,x,-)\colon \Rp\to X$. 
 From the construction, it is easy to see that 
 this map is a geodesic, and 
 $(\rpgm(e,x,t)\mid x)_e \to \infty$ as $t\to \infty$.
\end{proof}

\begin{definition}
 We call the map $\rpgm$ given in \cref{lem:g-extended_bicombing} 
 an \emph{extended geodesic bicombing} on 
 $X\times \overline{X}$ corresponding to $\gamma$.
 For $(e,x)\in X\times \partial X$, we abbreviate 
 $\rpgm(e,x,-)$ by $\gamma_e^x$.
\end{definition}

\begin{lemma}
\label{lem:g-rpgm_cconvex}
 The extended geodesic bicombing $\rpgm$ on $X\times \overline{X}$ satisfies
 the following:
        Let $(e_1,x_1),(e_2,x_2)\in X\times \partial X$ 
        and let $a,b \in \Rp$.
       Set $x_1':=\rpgm(e_1,x_1,a)$ and $x_2':=\rpgm(e_2,x_2,b)$.
       Then, for $c \in [0,1]$, we have
        \begin{align*} 
         d_X(\rpgm(e_1,x_1,ca), \rpgm(e_2,x_2,cb)) 
         \leq (1-c)Ed_X(e_1,e_2) + cEd_X(x_1',x_2') + C.
        \end{align*}
\end{lemma}

\begin{proof}
 The statement follows immediately 
 from~\ref{item:approximation} of \cref{lem:g-extended_bicombing}.
\end{proof}


\begin{remark}
 All statements in \cref{lem:univ-const} hold for $\gamma_e^x$.
\end{remark}

\begin{lemma}
\label{lem:gprod-nearpts}
 For $e,x,y\in X$, we have
 \begin{align*}
  (x\mid y)_e\geq \frac{\min\{d_X(e,x),d_X(e,y)\}}{2Ed_X(x,y)}.
 \end{align*}
\end{lemma}

\begin{proof}
 Set $u:=\min\{d_X(e,x),d_X(e,y)\}$. We have
 \begin{align*}
  d_X(\rp\gamma(e,x,u),\rp\gamma(e,y,u))\leq 
  d_X(x,y) + \abs{d_X(e,x)-d_X(e,y)} \leq 2d_X(x,y).
 \end{align*}
 Set $c:=1/(2Ed_X(x,y))$, we have
 \begin{align*}
  d_X(\rp\gamma(e,x,cu),\rp\gamma(e,y,cu))\leq 
  cEd_X(\rp\gamma(e,x,u),\rp\gamma(e,y,u)) +C
  \leq 1+C.
 \end{align*}
 Therefore 
 \begin{align*}
  (x\mid y)_e\geq cu = \frac{\min\{d_X(e,x),d_X(e,y)\}}{2Ed_X(x,y)}.
 \end{align*}
\end{proof}

\subsection{$\gamma$-convex subspaces}

\begin{definition}
 Let $X$ be a metric space with $(E,C)$-geodesic coarsely convex bicombing 
 $\gamma\colon X\times X \times [0,1]\to X$. We say that a subspace $Y\subset X$
 is $\gamma$-convex if for all $x,y\in Y$ and all $t\in [0,1]$, we have $\gamma(x,y,t)\in Y$.
\end{definition}

\begin{proposition}
\label{prop:gamma-convex-subsp}
 Let $X$ be a metric space with $(E,C)$-geodesic coarsely convex bicombing 
 and $Y\subset X$ be a $\gamma$-convex subspace. 
 Then the restriction $\gamma_Y$ of $\gamma$ on 
 $\colon Y\times Y \times [0,1]$ is a $(E,C)$-geodesic coarsely convex 
 bicombing on $Y$.
 Moreover, the inclusion $\iota\colon Y\hookrightarrow X$ extends to a topological embedding 
 $\bar{\iota}\colon \overline{Y}\hookrightarrow \overline{X}$.
\end{proposition}

\begin{corollary}
\label{cor:sub-sp-closure}
 Let $X$ be a metric space with $(E,C)$-geodesic coarsely convex bicombing 
 and $Y\subset X$ be a $\gamma$-convex subspace. Then the ideal boundary of $Y$ with respect to
 $\gamma_Y$ is homeomorphic to $\overline{\iota(Y)}\setminus \iota(Y)$, 
 where $\overline{\iota(Y)}$ is the closure of $\iota(Y)$ in $\overline{X}$.
\end{corollary}

\section{Continuous at infinity}
\label{sec:cont-at-infin}
We study maps between geodesic coarsely convex spaces
which are not necessarily continuous, although, induce continuous maps
between ideal boundaries. The main result of this section will be 
used to prove \cref{prop:retractionAn}, which is a key proposition 
for results in \cref{sec:cohomology-ideal-boundary-tree}

\begin{definition}
\label{def:conti-at-infty}
 Let $X$ be a space with a geodesic coarsely convex bicombing.
 Let $A\subset X$ be a subspace and $\overline{A}$ be the closure of 
 $A$ in $\overline{X}$.
 Let $U$ be a topological space.
 We say that a map $F\colon \overline{A}\to U$ is \emph{continuous at infinity} 
 if $F$ is continuous at all $x\in \overline{A}\setminus A$.
\end{definition}

Let $(X,d_X)$ be a space equipped with an
$(E,C)$-geodesic coarsely convex bicombing 
$\gamma\colon X\times X \times [0,1]\to X$.

\begin{lemma}
\label{lem:qiemb-conti-at-infinity}
 Let $A\subset X$ be a subspace and $\overline{A}$ be the closure of 
 $A$ in $\overline{X}$. 
Let 
 $f\colon \overline{A} \to \overline{X}$ be a map such that:
 \begin{enumerate}[label=(\arabic*)]
  \item $\sup\{d_X(x,f(x)):x\in A\}<\infty$,  and
  \item for $x\in \partial X\cap \overline{A}$, $f(x)=x$.
 \end{enumerate}
 Then $f$ is continuous at infinity.
\end{lemma}

\begin{proof}
 Let $x\in \partial X\cap \overline{A}$. 
 We show that $f$ is continuous at $x$. 
 Set $K:=\sup\{d_X(x,f(x)):x\in A\}$.
 Let $(x_n)$ be a sequence in $\overline{A}$
 converging to $x$. We can assume
 without loss of generality that $x_n\in A$ for all $n$.
 Set $t_n:=\min\{d_X(e,f(x_n)),d_X(e,x_n)\})$.
 Since $\abs{d_X(e,f(x_n))-d_X(e,x_n)}\leq K$, we have
\begin{align*}
 d_X(\rp\gamma(e,f(x_n),t_n),\rp\gamma(e,x_n,t_n))\leq 
 K + d_X(x_n,f(x_n)) \leq 2K.
\end{align*} 
Set
\begin{align*}
  p_n&:= \rp\gamma(e,f(x_n),(1/(2KE)t_n)),\\
  q_n&:= \rp\gamma(e,x_n,(1/(2KE)t_n)).
 \end{align*}
 We have
 \begin{align*}
  d_X(p_n,q_n)
  &\leq \frac{1}{2KE}E
  d_X(\rp\gamma(e,f(x_n),t_n),\rp\gamma(e,x_n,t_n))
  + C \leq C+1
 \end{align*}
 So we have $(x_n\mid f(x_n))_{e} \geq 1/(2KE)t_n$.
 Then,
 \begin{align*}
  {(x\mid f(x_n))}_e\geq 
  \Omega^{-1}\min\{{(x\mid x_n)}_{e}, {(x_n\mid f(x_n))}_{e}\}
  \to \infty
 \end{align*}
 It follows that $(f(x_n))$ converges to $x$.
\end{proof}

\subsection{Image of the exponential map}
We fix a base point $e\in X$.
We define a map $\exp \colon \partial X \times \Rp \to X$ by
$\exp(x,t):= \rpgm(e,x,t)$. 
Set 
\begin{align*}
 \exp(\OdX) &:= \{\rpgm(e,x,t):x\in \partial X_k, t\in \Rp\}\\
 \overline{\exp(\OdX)} &:= \exp(\OdX)\cup \partial X_k.
\end{align*}
We remark that $\exp(\OdX)$ is not necessarily coarsely dense
in $X$. In~\cite[Section 5.5]{FO-CCH}, 
a coarsely equivalence map $\tilde{\varphi}\colon X \to \exp(\OdX_k)$ is 
contracted. We review the construction.

Let $X^{(0)}$ be a subset of $X$ such that
$X^{(0)}$ is 2-dense in $X$ and 1-discrete, that is,
for all $v,w\in X^{(0)}$, if $v\neq w$
then $d_X(v,w)\geq 1$, 
and, for all $v\in X$, there exists $v'\in X^{(0)}$ with $\ds{v,v'}\leq 2$.
Set $D_6:=2D_1E+D$ and $Y:= B_{D_6}(\exp(\OdX))$. 

Let $\chi\colon \Rp\to \Rp$ is a proper map given 
in~\cite[Section 5.5]{FO-CCH}. By the construction of $\chi$, for
$v\in X^{(0)}$, we have $\rp\gamma(e,v,\chi(d_X(e,v)))\in Y$,
so we define a map
$\varphi\colon X^{(0)}\to Y$ by
\begin{align*}
 \varphi(v):=\rp \gamma(e,v,\chi(d_X(e,v))).
\end{align*}

We define a map 
$\overline{\varphi}\colon \overline{X^{(0)}}\to \overline{Y}$ by
\begin{align*}
 \overline{\varphi}(v):=
 \begin{cases}
  \varphi(v) & \text{ if } v\in X^{(0)}\\
  v & \text{ if } v\in \partial X
 \end{cases}.
\end{align*}

Here we remark that $\overline{X^{(0)}}\setminus X^{(0)} = \partial X$
since $X^{(0)}$ is 2-dense in $X$.

\begin{lemma}
\label{lem:varphi-conti}
 The map $\overline{\varphi}\colon \overline{X^{(0)}}\to \overline{Y}$
 is continuous at infinity.
\end{lemma}

\begin{proof}
 Let $x\in \partial X$. We show that $\overline{\varphi}$ is continuous
 at $x$. Let $(x_n)$ be a sequence in $\overline{X^{(0)}}$
 converging to $x$. We can assume
 without loss of generality that $x_n\in X^{(0)}$ for all $n$.
 We remark $(x\mid x_n)_e\to \infty$.
 By \ref{gprod-contraction} of \cref{lem:univ-const}, we have
 \begin{align*}
  (x\mid \varphi(x_n))_e\geq \Omega 
  \min\{(x\mid x_n), \chi(d_X(e,x_n))\}
  \to \infty
 \end{align*}
 since $\chi$ is proper. It follows that $(\varphi(x_n))$ 
 converges to $x$.
\end{proof}

Let $\iota\colon X\to X^{(0)}$ be a restriction
such that for $x\in X$, we have 
$d_X(x,\iota(x))\leq 2$. 
Let $j\colon Y\to \exp(\OdX)$ be a retraction
such that, for $x\in Y$, we have 
$d_X(x,j(x))\leq D_6$. Here we do not require that $i$ and $j$ are continuous.
We define a map 
$\overline{\iota}\colon \overline{X}\to \overline{X^{(0)}}$ and 
$\overline{j}\colon \overline{Y}\to \overline{\exp(\OdX)}$ 
by
\begin{align*}
 \overline{\iota}(v):=
 \begin{cases}
  \iota(v) & \text{ if } v\in X^{(0)}\\
  v & \text{ if } v\in \partial X
 \end{cases} 
 \quad \text{and} \quad
 \overline{j}(v):=
 \begin{cases}
  j(v) & \text{ if } v\in Y\\
  v & \text{ if } v\in \partial X 
 \end{cases}.                   %
\end{align*}

By \cref{lem:qiemb-conti-at-infinity},
the maps $\overline{\iota}$ and $\overline{j}$ 
are continuous at infinity.
We define a map $\Phi\colon \overline{X}\to \overline{\exp(\OdX)}$ by
\begin{align*}
 \Phi(x):=\begin{cases}
           \overline{j}\circ \overline{\varphi}\circ \overline{\iota}(x)
           & \text{if } x\in X\\
           x &  \text{if } x\in \partial X
          \end{cases}.
\end{align*}
\begin{lemma}
\label{prop:deform-conti}
 The map $\Phi\colon \overline{X}\to \overline{\exp(\OdX)}$ 
 is continuous at infinity.
\end{lemma}

\subsection{Angle map}
\label{sec:angle-map}
We define a map $\arg\colon \overline{\exp(\OdX)}\to \partial X$ 
as follows:
For $x\in \exp(\OdX_k)$, we choose $(p_x,t_x)\in \partial X_k\times \Rp$ such that
$x=\rpgm(e,p_x,t_x)$.
Then we define
\begin{align*}
 \arg(x) = p_x.
\end{align*}

\begin{lemma}
\label{lem:arg-conti}
 The map $\arg\colon \exp(\OdX_k)\to \partial X_k$ is 
 continuous at infinity.
\end{lemma}
We define a map $\Psi\colon \overline{X}\to \partial X$ 
\begin{align*}
 \Psi(x):=\begin{cases}
           \arg\circ \Phi(x)
           & \text{if } x\in X\\
           x &  \text{if } x\in \partial X
          \end{cases}
\end{align*}
By~\cref{prop:deform-conti,lem:arg-conti}, we have the following.
\begin{proposition}
 \label{prop:Psi-conti-at-infty}
 The map $\Psi\colon \overline{X}\to \partial X$ is 
 continuous at infinity.
\end{proposition}

\section{Trees of spaces}
\label{sec:trees-spaces}
\subsection{Trees of spaces}
\label{sec:tree-spaces}
\begin{definition}
\label{def:tree-of-spaces}
 Let $(Z,d_Z)$ be a geodesic metric space and $T$ be a bipartite tree.
 Set $V(T):= K \sqcup L$. 
 Let $\xi\colon Z\to T$ be a continuous map.
 We suppose the following:
 \begin{enumerate}[label=(\arabic*)]
  \item For each $l\in L$, the inverse image $\xi^{-1}(l)$ consists of a
        single point.
  \item For each $k\in K$, the inverse image $X_k:=\xi^{-1}(\starNb(k))$
        is a geodesic space.
 \end{enumerate}
 Then we say that $Z$ is a \emph{tree of spaces
 associated with the map $\xi\colon Z\to T$ and the bipartite
 structure $V(T)=K \sqcup L$}.
\end{definition}

\begin{definition}
\label{def:tree-of-gcc}
 Let $(Z,d_Z)$ be a tree of spaces
 associated with the map $\xi\colon Z\to T$ and the bipartite
 structure $V(T)=K \sqcup L$.
 
 Let $E'\geq 1$ and $C'\geq 0$ be constants.  We
 suppose that for each $k\in K$, the inverse image $X_k:=\xi^{-1}(\starNb (k))$
 admits an $(E',C')$-geodesic coarsely convex bicombing 
 \begin{align*}
  \gamma_k\colon X_k\times X_k\times [0,1]\to X_k.
 \end{align*} 
 Then we say that $Z$ is a 
 \emph{$(\xi,K,L,\{\gamma_k\}_{k\in K})$-tree of geodesic coarsely convex spaces}. 
\end{definition}

\begin{definition}
\label{def:tree-of-gcc-nolabel}
 We say that a metric space $Z$ is \emph{a tree of 
 geodesic coarsely convex spaces} if 
 $Z$ is $(\xi,K,L,\{\gamma_k\}_{k\in K})$-tree of 
 geodesic coarsely convex space for some $(\xi,K,L,\{\gamma_k\}_{k\in K})$.
\end{definition}


\begin{theorem}[\cite{gcc-pair-preparation}]
\label{thm:bicombing-on-Z}
 Let $(Z,d_Z)$ be a $(\xi,K,L,\{\gamma_k\}_{k\in K})$-tree 
 of geodesic coarsely convex space.
 For $k\in K$, set $X_k:=\xi^{-1}(\starNb (k))$.

 Then the space $Z$ admits an  
 $(E,C)$-geodesic coarsely convex bicombing 
 \begin{align*}
  \gamma\colon Z\times Z\times [0,1]\to Z
 \end{align*}
 where $E=E'+2$ and $C=6C'$. 
 The restriction of $\gamma$ to $X_k\times X_k\times [0,1]$ is 
 equal to $\gamma_k$.
\end{theorem}

\begin{definition}
 Let $(Z,d_Z)$ be a $(\xi,K,L,\{\gamma_k\}_{k\in K})$-tree 
 of geodesic coarsely convex space.
 We call the bicombing $\gamma$ on $Z$ constructed in
 \cref{thm:bicombing-on-Z} \emph{the geodesic coarsely convex bicombing
 associated with the family $\{\gamma_k\}$}.
\end{definition}

\begin{remark}
\label{rem:dX_k-emb}
 Let $(Z,d_Z)$ be a $(\xi,K,L,\{\gamma_k\}_{k\in K})$-tree 
 of geodesic coarsely convex space, and let $\gamma$
 be the geodesic coarsely convex bicombing on $Z$
 associated with the family $\{\gamma_k\}$.
 For $k\in K$, set $X_k:=\xi^{-1}(\starNb (k))$.

 By the construction, it is clear that for each $k\in K$,
 the subspace $X_k$ is $\gamma$-convex in $Z$. 
 So by \cref{cor:sub-sp-closure}, the boundary of $X_k$ in 
 $\overline{Z}$ is homeomorphic to $\partial X_k$.
\end{remark}



\subsection{Free products of metric spaces}
In \cite{fukaya2023free}, we introduce free products of metric spaces, 
which can be regarded as trees of metric spaces. 
\begin{definition}
\label{def:net}
 Let $(X,d_X)$ be a metric space and let $X_0$ be a set with the discrete topology.
 Let $i_X: X_0 \to X$ be a proper map.
 We choose a \textit{base point} $e_X\in i_X(X_0)$.
 We call $(X_0, i_X,e_X)$ a \textit{net} of $X$.
 For  $x_0 \in X_0$, we denote by $\overline{x_0}$ the image $i_X(x_0)$.
 We say that $(X,d_X,X_0,i_X,e_X)$ is \emph{a metric space with a net}.
\end{definition}

\begin{remark}
 We often abbreviate specifying a net $(X_0,i_X,e_X)$ and say that 
 $(X,d_X)$ is a metric space with a net.
\end{remark}

\begin{example}
 \label{eg:net}
 Let $(X,d_X,e_X)$ be a metric space with a base point $e_X$.
 Let $G$ be a group acting on $X$ properly by isometries. 
 We define a map $\mathsf{orb}\colon G\to X$ by $\mathsf{orb}(g)\coloneqq g\cdot e_X$.
 Then $(G, o(e_X),e_X)$ is a net of $X$, called \textit{G-net}.
\end{example}

Let $(X,d_X)$ and $(Y,d_Y)$ be metric spaces with nets. In \cite[Section 1]{fukaya2023free},
we constructed the free product $X * Y$ of $X$ and $Y$.
In \cite{gcc-pair-preparation}, we constructed a bipartite tree $T$ and continuous map 
$\xi\colon X*Y\to T$ such that $X*Y$ is a tree of spaces associated with $\xi$.

\begin{example}
 \label{ex:freeprod-groups}
 Let $G$ and $H$ be groups acting on metric spaces $X$ and $Y$. We consider 
 a $G$-net of $X$, and an $H$-net of $Y$, respectively. Then the free product 
 $G * H$ acts on $X*Y$ properly. If the action of $G$ and $H$ are cocompact,
 then so is the action of $G*H$.
\end{example}

\section{Augmented space}
\label{sec:augmented-space}
\subsection{$n$-th augmented spaces}

\begin{definition}
\label{def:n-aug}
 Let $(Z,d_Z)$ be a tree of spaces
 associated with the map $\xi\colon Z\to T$ and the bipartite
 structure $V(T)=K \sqcup L$.

 For $k\in K$, set $X_k\colon=\xi^{-1}(\starNb (k))$. We choose an
 1-discrete, 2-dense subset  $X_k^{(0)}\subset X_k$.
 So the pair $(X_k, X_k^{(0)})$ is a space with a lattice. 
 We choose a bijection $k\colon \N\to K$.

 Let $n\in \N\sqcup \{0\}$.
 The \emph{$n$-th augmented space $\augn{Z}{n}$ of $Z$}
 is a tree of spaces
 associated with the map $\augn{\xi}{n}\colon \augn{Z}{n}\to T$ 
 satisfying the following:
 \begin{enumerate}[label=(\arabic*)]
  \item For $i\in \N$ with $i\leq n$, the space 
        $\augn{\xi}{n}^{-1}(\starNb (k(i)))$ is isometric to $X_{k(i)}$.
  \item For $i\in \N$ with $i> n$, 
        the space $\augn{\xi}{n}^{-1}(\starNb (k(i)))$ 
        is isometric to the metric horoball 
        $\Horo(X_{k(i)},X_{k(i)}^{(0)})$ of the space with a lattice $(X_{k(i)},X_{k(i)}^{(0)})$.
 \end{enumerate}
 We call $0$-th augmented space $\augn{Z}{0}$
 \emph{the augmented space}, and we denote by $\aug{Z}$.
\end{definition}

\begin{remark}
 In the above definition, the coarse equivalence class of 
 the $n$-th augmented space does not depend on the choice of 
 nets $X_k^{(0)}$. Therefore, when we consider $n$-th augmented spaces,
 we implicitly choose nets, and we say that $\augn{Z}{n}$ is the
 $n$-th augmented space of $Z$, without mentioning nets.
 We abbreviate $\Horo(X_k,X_k^{(0)})$ by $\Horo(X_k)$

 The augmented space $\aug{Z}$ does not depend on 
 the choice of the bijection $k\colon \N\to K$.
\end{remark}

We can explicitly construct a $n$-th augmented space $\augn{Z}{n}$
from a tree of spaces $Z$
by choosing nets and gluing combinatorial horoballs along nets
by the way described in \cref{def:metr-horob}.

\subsection{Bicombings on $n$-th augmented spaces}
Let $(Z,d_Z)$ be a $(\xi,K,L,\{\gamma_k\}_{k\in K})$-tree 
of geodesic coarsely convex spaces.
Let $\gamma$ be the geodesic coarsely convex bicombing
associated with the family $\{\gamma_k\}$.
We choose a pair of constants $(E',C')$ such that
$\gamma$ is $(E',C')$ geodesic coarsely convex.

For $k\in K$, set $X_k\colon=\xi^{-1}(\starNb (k))$. 
For $k\in K$, we define an  
$(E',C'+8\delta)$-geodesic coarsely convex bicombing $\augn{\gamma_{k}}{n}$
on $\augn{\xi}{n}^{-1}(\starNb (k))=\augn{X_{k}}{n}$ 
by the following:
\begin{enumerate}
 \item if $i\leq n$ then $\augn{\gamma_{k(i)}}{n} = \gamma_{k(i)}$,
 \item Suppose $i > n$.
 Note that $\augn{X_{k(i)}}{n}$ is equal to $\Horo(X_{k(i)})$. 
       For $x,y \in \augn{X_{k(i)}}{n}$, we define a map
 \begin{align*}
\augn{\gamma_{k(i)}}{n}(x,y,-)\colon [0,1]\to \augn{X_{k(i)}}{n}
 \end{align*}
       to be the normal geodesic segment from $x$ to $y$ 
       given in \cref{prop:geod}.
\end{enumerate}

We choose a bijection $k\colon \N\to K$.
Let $\augn{Z}{n}$ be the $n$-th augmented space of $Z$ .
Then $\augn{Z}{n}$ is a $(\augn{\xi}{n},K,L,\{\augn{\gamma_k}{n}\}_{k\in K})$-tree 
of geodesic coarsely convex space.
Let $\augn{\gamma}{n}$ be an 
geodesic coarsely convex bicombing on $\augn{Z}{n}$ 
associated with $\{\augn{\gamma_{k}}{n}\}$.

\begin{definition}
\label{def:bicombing_on_augn}
 We call the bicombing $\augn{Z}{n}$ constructed in the above
 an \emph{augmented bicombing associated with the family 
 $\{\gamma_{k}\}_{k\in K}$}.
\end{definition}

\begin{remark}
 For $i\leq n$, $X_{k(i)}$ is an $\augn{\gamma}{n}$-convex subspace, 
 so by \cref{cor:sub-sp-closure}, $\partial X_{k(i)}$ 
 is a closed subspace of $\partial\augn{Z}{n}$.

 For $i>n$, $\Horo(X_{k(i)})$ is an $\augn{\gamma}{n}$-convex subspace, the boundary 
 $\partial \Horo(X_{k(i)}) = \overline{\Horo(X_{{k(i)}})} \setminus \Horo(X_{k(i)})$ 
 consists of a single point. We call this the 
 \emph{center of the horoball} $\Horo(X_{k(i)})$, 
 and denote by $c_{k(i)}$.
\end{remark}

\section{(co)Homology if Ideal boundary of a tree of spaces}
\label{sec:cohomology-ideal-boundary-tree}
Let $(Z,d_Z)$ be a $(\xi,K,L,\{\gamma_k\}_{k\in K})$-tree 
of geodesic coarsely convex spaces.
We fix a bijection $k\colon \N\to K$.
Let $\aug{Z}$ be the augmented space of $Z$, and
let $\augn{\gamma}{n}$ be the augmented bicombing on $\augn{Z}{n}$
associated with the family $\{\gamma_{k}\}_{k\in K}$.

\begin{remark}
 It is clear that the inclusion 
 $\iota_{n}\colon \augn{Z}{n}\hookrightarrow \augn{Z}{n-1}$ is a visual map 
 in the sense of \cref{def:visual-map}.
 By \cref{prop:visual-df-conti}, it induces a continuous map 
 \begin{align*}
 \partial \iota_{n}\colon \partial \augn{Z}{n} \to  \partial\augn{Z}{n-1}.
 \end{align*}
 The family $\{\partial \augn{Z}{n}, \partial \iota_{n}\}$ forms a projective system,
 and the projective limit is $\partial Z$, that is,
 \begin{align*}
  \partial Z = \varprojlim_{n} \partial \augn{Z}{n}.
 \end{align*}
\end{remark}

The following is a key proposition for computing (co)homologies 
of $\partial Z$.
\begin{proposition}
\label{prop:retractionAn}
 For $n\in \N$, there exists a continuous retraction 
 \begin{align*}
  s_{k(n)}\colon \partial \augn{Z}{n}\to \partial X_{k(n)}.
 \end{align*}
\end{proposition}
We will prove \cref{prop:retractionAn} in \cref{sec:constr-retarct}

\subsection{Cohomology}
\label{sec:cohomology}
Let $M^*=(M^n)_{n\in \N}$ be the $K$-theory or the 
Alexander-Spanier cohomology, and let $\tilde{M}^*$ be the reduced one.
For a compact Hausdorff space $Y$, and $e\in Y$, we have
$\tilde{M}^*(Y) \cong M^*(Y,\{e\})$.


We consider the long exact sequence 
for the triple
$(\partial \augn{Z}{n}, \partial X_{k(n)},\{e_n\})$.
\begin{align*}
  \dots &\to M^p(\partial \augn{Z}{n},\partial X_{k(n)}) 
 \to M^p(\partial \augn{Z}{n},\{e_n\})
       \to M^p(\partial X_{k(n)},\{e_n\}) \\
 &\to M^{p+1}(\partial \augn{Z}{n},\partial X_{k(n)}) \to \cdots
\end{align*}

By \cref{prop:retractionAn}, the above exact sequence splits, so we have
\begin{align}
\label{eq:cohoml-splits}
  0 \to M^p(\partial \augn{Z}{n},\partial X_{k(n)}) 
 \to \tilde{M}^p(\partial \augn{Z}{n})
       \to \tilde{M}^p(\partial X_{k(n)}) \to 0
\end{align}
By the strong excision axiom, we have  
\begin{align*}
 M^p(\partial \augn{Z}{n},\partial X_{k(n)})&\cong 
 M^p((\partial \augn{Z}{n}\setminus \partial X_{k(n)})^{+},\{c_{k(n)}\})\\
 &\cong
 M^p(\partial \augn{Z}{n-1},\{c_{k(n)}\}) \\
 &= \tilde{M}^p(\partial \augn{Z}{n-1}).
\end{align*}
Here $(\partial \augn{Z}{n}\setminus \partial X_{k(n)})^{+}$
 denotes the one-point compactification 
\begin{align*}
 (\partial \augn{Z}{n}\setminus \partial X_{k(n)})\sqcup \{c_{k(n)}\}.
\end{align*}

It follows that 
$\tilde{M}^p(\partial \augn{Z}{n}) \cong \tilde{M}^p(\partial \augn{Z}{n-1})  \oplus \tilde{M}^p(\partial X_{k(n)})$.
By the continuity of $M^*$, we have
$\tilde{M}^p(\partial Z) \cong \varinjlim \tilde{M}^p(\partial \augn{Z}{n})$.
Here we summarize the computation.
\begin{theorem}
\label{thm:cohomology}
Let $Z$ be a 
$(\xi,K,L,\{\gamma_k\}_{k\in K})$-tree of geodesic coarsely convex spaces.
Suppose that $Z$ is proper.
Set $X_k:=\xi^{-1}(\starNb (k))$. Let $\aug{Z}$ be the augmented space of $Z$. Then,
 \begin{align*}
 \tilde{M}^p(\partial Z) 
 \cong \tilde{M}^p(\partial \aug{Z}) \oplus 
  \bigoplus_{k\in K} \tilde{M}^p(\partial X_k)
 \end{align*}
\end{theorem}
Here we remark that the topological type of $\partial \aug{Z}$ is described in
\cref{thm:Cantor}.

\subsection{Homology}
\label{sec:homology} 
Let $M_*=(M_n)_{n\in \N}$ be the $K$-homology or the Steenrod homology,
and let $\tilde{M}_*$ be the reduced one.
For a compact Hausdorff space $X$, and $e\in X$, we have
$\tilde{M}_*(X) \cong M_*(X,\{e\})$.


\begin{theorem}[{\cite[Theorem 4.]{MR1388297}}]
\label{thm:Milnor-exact}
 Let $Y_1\leftarrow Y_2\leftarrow Y_3\leftarrow \cdots$ be an inverse system of compact metric spaces.
 Then there exists an exact sequence
\begin{align*}
 0\to \limone M_{p+1}(Y_i) \to M_p(\varprojlim Y_i) \to \varprojlim M_p(Y_i) \to  0.
\end{align*}
\end{theorem}

By the same argument as the one in \cref{sec:cohomology}, we obtain
\begin{align*}
 \tilde{M}_p(\partial \augn{Z}{n}) \cong \tilde{M}_p(\partial \augn{Z}{n-1})  \oplus \tilde{M}_p(\partial X_{k(n)}).
\end{align*}

By \cref{thm:Milnor-exact}, we have
\begin{align*}
 0\to \limone M_{p+1}(\partial \augn{Z}{n}) 
 \to M_p(\partial Z) \to \varprojlim M_p(\partial \augn{Z}{n}) \to  0.
\end{align*}

By \cref{prop:retractionAn}, the maps 
$M_{p+1}(\partial \augn{Z}{n+1})\to M_{p+1}(\partial \augn{Z}{n})$ are surjective,
so the Mittag-Leffler condition holds.
Thus $\limone M_{p+1}(\partial \augn{Z}{n})$ vanishes. 
\begin{theorem}
\label{thm:homology}
 Let $Z$ be a $(\xi,K,L,\{\gamma_k\}_{k\in K})$-tree of 
 geodesic coarsely convex spaces. Set $X_k:=\xi^{-1}(\starNb (k))$.
 Suppose that $Z$ is proper.
 Let $\aug{Z}$ be the augmented space of $Z$. Then,
 \begin{align*}
 \tilde{M}_p(\partial Z) 
 \cong \tilde{M}_p(\partial \aug{Z}) \times
  \prod_{k\in K} \tilde{M}_p(\partial X_k).
 \end{align*}
\end{theorem}

\subsection{$K$-theory of the Roe algebra}
Let $V$ be a proper metric space. We denote by $C^*(V)$ the Roe algebra
of $V$. We also denote by $KX_*(V)$ the coarse $K$-homology of $V$.
For detail, see \cite{MR1388312}, \cite{MR1344138} and
\cite{MR1817560}.

Suppose that $V$ is a proper coarsely convex space.
By \cite[Theorem 1.3]{FO-CCH}, the following coarse assembly map
is an isomorphism.
\begin{align*}
 \mu_{*} \colon KX_*(V)\to K_*(C^*(V)).
\end{align*}

By \cite[Theorem 6.7]{FO-CCH}, the following transgression map 
is an isomorphism.
\begin{align*}
 {T_{\partial V}} \colon KX_*(V) \to \tilde{K}_{*-1}(\partial V).  
\end{align*}
Thus we have an isomorphism
\begin{align}
\label{eq:K-Roe-b-map}
 {T_{\partial V}}\circ \mu_{*}^{-1}\colon K_*(C^*(V)) 
 \to \tilde{K}_{*-1}(\partial V).  
\end{align}

\begin{theorem}
\label{thm:K-Roe-alg-bdry-tree}
 Let $Z$ be a $(\xi,K,L,\{\gamma_k\}_{k\in K})$-tree of 
 geodesic coarsely convex spaces. Set $X_k:=\xi^{-1}(\starNb (k))$.
 Suppose that $Z$ is proper.
 Let $\aug{Z}$ be the augmented space of $Z$. Then,
 \begin{align*}
  K_p(C^*(Z)) 
 &\cong \tilde{K}_{p-1}(\partial \aug{Z}) \times
  \prod_{k\in K} \tilde{K}_{p-1}(\partial X_k)
\end{align*} 
\end{theorem}
\begin{proof}
 Combining \cref{thm:homology}
 for $K$-homology with an isomorphism 
 \begin{align*}
  K_*(C^*(Z))\cong \tilde{K}_{*-1}(\partial Z)
 \end{align*}
 obtained by applying \cref{eq:K-Roe-b-map} for 
 geodesic coarsely convex space $Z$, we obtain the desired result.
\end{proof}

\begin{corollary}
 \label{cor:K-Roe-alg-tree}
 Let $Z$ be the tree of geodesic coarsely convex spaces
 described in \cref{thm:K-Roe-alg-bdry-tree}.
 We have
 \begin{align*}
  K_p(C^*(Z)) 
 &\cong \tilde{K}_{p-1}(\partial \aug{Z}) \times
  \prod_{k\in K} K_{p}(C^*(X_k))
\end{align*} 
\end{corollary}

\begin{proof}
 Combining \cref{thm:K-Roe-alg-bdry-tree} with isomorphisms
 \begin{align*}
  K_*(C^*(X))\cong \tilde{K}_{*-1}(\partial X) \quad \text{and}\quad
  K_*(C^*(Y))\cong \tilde{K}_{*-1}(\partial Y),
 \end{align*}
 we obtain the desired result.
\end{proof}

\subsection{Case of free products}

\begin{proof}[Proof of \cref{thm:cohomology_free_prod,thm:K-Roe-alg-bdry,cor:K-Roe-alg}]
 Let $X$ and $Y$ be proper geodesic coarsely convex spaces with nets.
 Suppose that the net of $X$ and that of $Y$ are coarsely dense
 in $X$ and $Y$, respectively.
 Let $X * Y$ be the free product of $X$ and $Y$.
 Then $X * Y$ is a tree of space. Since 
 the net of $X$ and $Y$ are coarsely dense in $X$ and $Y$, respectively,
 by \cref{thm:Cantor}, the ideal boundary $\partial \aug{X*Y}$ of 
 the augmented space has no isolated point. 
 So, $\partial \aug{X*Y}$ is homeomorphic to the Cantor space 
 $\mathcal{C}$. Then applying 
 \cref{thm:cohomology,thm:homology,thm:K-Roe-alg-bdry-tree,cor:K-Roe-alg-tree} for $X*Y$,
 we obtain the desired result.
\end{proof}

\section{Construction of retractions}
\label{sec:constr-retarct}
Let $(Z,d_Z)$ be a $(\xi,K,L,\{\gamma_k\}_{k\in K})$-tree 
of geodesic coarsely convex space. We suppose that $Z$ is proper.
Let $\gamma$ be the geodesic coarsely convex bicombing on $Z$
associated with the family $\{\gamma_k\}_{k\in K}$.
By \cref{rem:dX_k-emb}, we can regard
$\partial X_k$ as a closed subspace of $\partial Z$.
The purpose of this section is to construct a retraction
$s_k\colon \partial Z\to \partial X_k$.

Let $k\in K$. For $x\in \partial Z$, set
\begin{align*}
 t(k;x):= \sup\{t\in \R: \rpgm(e,x,t)\in X_k\}. 
\end{align*}
We define a map $\pi_k\colon \partial Z\to \overline{X_k}$ by
\begin{align*}
 \pi_k(x):=\rpgm(e,x,t(k;x)).
\end{align*}

\begin{remark}
 For $x\in \partial Z\setminus \partial X_k$, we have 
 $t(k;x) = d(e,\rpgm(e,x,t(k;x)))$.
 For $x\in \partial X_k$, we have $t(k;x) = \infty$ and 
 $\pi_k(x)=\rpgm(e,x,\infty) = x$.
\end{remark}

\begin{lemma}
 \label{lem:estimate_Gprod1}
 We fix $k\in K$. Let $e\in X_k$ be a base point of $X_k$, and let $v,w\in X_k$ such that
 $v\neq w$. 
 Then for $x\in \pi_k^{-1}(v)$ and $y\in \pi_k^{-1}(w)$, we have
 \begin{align*}
  (x\mid y)_e < \Omega(\max\{d_Z(e,v), d_Z(e,w)\} + \Omega).
 \end{align*}
 Here $\Omega$ is a constant given in \cref{lem:univ-const}.
\end{lemma}

\begin{proof}
 We define $t\in \R$ by
 \begin{align*}
  t=\max\{t(k;x),t(k;y)\} + \Omega = \max\{d_Z(e,v),d_Z(e,w)\} + \Omega
 \end{align*}
 The geodesic from $\gamma_e^x(t)$ to $\gamma_e^y(t)$ is obtained by 
 concatenating $\gamma(\gamma_e^x(t),v,-)$, $\gamma(v,w,-)$ and $\gamma(w,\gamma_e^y(t),-)$.
 So we have
 \begin{align*}
  d_Z(\gamma_e^x(t),\gamma_e^y(t)) &= d_Z(\gamma_e^x(t),v) + d_Z(v,w) + 
  d_Z(w,\gamma_e^y(t))\\
  &\geq  t- t(k;x) + d_Z(v,w) + t-t(k;y)\\
  &\geq \Omega. 
 \end{align*}
 Thus by \ref{lem:maximizer} of \cref{lem:univ-const}, we have
 $(\gamma_e^x\mid \gamma_e^y)_e < t$. Then by \ref{lem:qultmD} of \cref{lem:univ-const}, we 
 have $(x\mid y)_e\leq \Omega (\gamma_e^x\mid \gamma_e^y)_e \leq \Omega t$.
\end{proof}

\begin{lemma}
 \label{lem:estimate_Gprod2}
 We fix $k\in K$. Let $e\in X_k$ be a base point of $X_k$, and let $v\in X_k$.
 Then for $x\in \partial X_k$ and $y\in \pi_k^{-1}(v)$, we have
 \begin{align*}
  (x\mid y)_e < \Omega(d_Z(e,v) + \Omega).
 \end{align*}
 Here $\Omega$ is a constant given in \cref{lem:univ-const}.
\end{lemma}

\begin{proof}
We can prove the lemma by a similar argument as 
the proof of \cref{lem:estimate_Gprod1}.
\end{proof}

\begin{lemma}
 The map $\pi_k\colon \partial Z\to \overline{X_k}$ is continuous.
\end{lemma}

\begin{proof}
 Let $x\in \partial Z$. We will show that $\pi_k$ is continuous at $x$.
 First, we suppose that $x \in \partial Z\setminus \partial X_k$. Set $v=\pi_k(x)$.
 Then by \cref{lem:estimate_Gprod1}, $\pi_k^{-1}(v)$ is an open set. This implies that
 $\pi_k$ is continuous at $x$.

 Now we suppose that $x\in \partial X_k$. We remark that $\pi_k(x)=x$.
 For $r\in \Rp$, set 
 $V_r[x]:=\{y\in \overline{X}:(x\mid y)> r\}$. 
 The family $\{V_n[x]:n\in \N\}$ forms a system of 
 the fundamental neighborhood of $x$.
 We will show that for each $n\in \N$, there exists $T\in \Rp$ 
 such that
\begin{align}
 \label{eq:pi_k}
  \pi_k(V_T[x])\subset V_{n-1}[\pi_k(x)] = V_{n-1}[x].
\end{align}

 We fix $n\in \N$. Set $T=EC\Omega(n+\Omega)$. 
 Let $y\in V_{T}[x]$.
 Thus $(x\mid y)_e > T$. Set $v=\pi_k(y)$. 
 By \cref{lem:estimate_Gprod2}, we have 
 \begin{align*}
  t(k;y) = d_Z(e,v)\geq \Omega^{-1}(x\mid y)_e -\Omega > ECn.
 \end{align*}
 By \ref{item:geod-bicombing-at-infty} of \cref{lem:univ-const},
 we have $d_Z(\gamma_e^x(T),\gamma_e^y(T))\leq \Omega$.
 Set $T'\coloneqq ECn$.
 Then we have
 \begin{align*}
  d_Z(\gamma_e^y(T'),\rp\gamma(e,v,T')) 
  &\leq Ed_Z(v,v) +C = C
 \end{align*}
 Here we used that $\gamma_e^y(t(k;y))=v=\rp\gamma(e,v,t(k;y)))$.
We have
 \begin{align*}
  d_Z(\gamma_e^x(n),\rp\gamma(e,v,n)) 
  &\leq \frac{1}{EC}E d_Z(\gamma_e^y(T'),\rp\gamma(e,v,T')) + C\\
  &\leq C+1. 
 \end{align*}
 It follows that $(x\mid v)\geq n$. 
 Thus $v=\pi_k(y)\in V_{n-1}[x]$. Therefore \cref{eq:pi_k} holds.
 This implies that $\pi_k$ is continuous at $x$.
\end{proof}

Let $\Psi_k\colon \overline{X_k}\to \partial X_k$ be the map
constructed in \cref{sec:angle-map}.

\begin{definition}
 We define a retraction $s_k\colon \partial Z\to \partial X_k$ by
 \begin{align*}
  s_k = \Psi_k \circ \pi_k.
 \end{align*}
\end{definition} 

\begin{proposition}
\label{prop:retractionZ}
 The map $s_k\colon \partial Z\to \partial X_k$ is continuous.
\end{proposition}

\begin{proof}
 Let $x\in \partial Z$. We will show that $\pi_k$ is continuous at $x$.


 Let $x_n$ be a sequence in $\partial Z$ converging to $x$. 
 Since $\pi_k$ is continuous, the sequence $\pi_k(x_n)$ converges
 to $\pi_k(x)\in \partial X_k$. Then by \cref{prop:Psi-conti-at-infty},
 the sequence $(\Psi_k(\pi_k(x_n)))$ converges to 
 $\Psi_k(\pi_k(x))$.
\end{proof}

\begin{proof}[Proof of \cref{prop:retractionAn}]
 Applying \cref{prop:retractionZ} to the tree of spaces
 $\augn{Z}{n}$ associated with the map
 $\augn{\xi}{n}\colon \augn{Z}{n}\to T$, 
 we obtain the desired result.
\end{proof}

\section{Ideal boundaries of the augmented spaces}
\label{sec:ideal-bound-augm}
Let $(Z,d_Z)$ be a $(\xi,K,L,\{\gamma_k\}_{k\in K})$-tree 
of geodesic coarsely convex space. We suppose that $Z$ is proper.


Let $\aug{Z}$ be the augmented space of $Z$. For simplicity,
we denote by $\gamma$ an augmented bicombing on $\aug{Z}$
associated with the family $\{\gamma_k\}_{k\in K}$.

For $x\in \overline{\aug{Z}}$, set 
$W_n[x] =\{y\in \partial \aug{Z}: (\gamma_e^x \mid \gamma_e^y)_e \geq n\}$. 
Recall that $\cbHoro(X_k^{(0)})^{(0)}$ denotes the vertex set of 
the combinatorial horoball $\cbHoro(X_k^{(0)})$.
We consider the family
\begin{align*}
 \mathcal{T}:=\bigsqcup_{k\in K} \{W_n[x]: x\in \cbHoro(X_k^{(0)})^{(0)}, n\in \N\}.
\end{align*}
Since $Z$ is proper, for any bounded subset 
$B\subset Z$, the cardinality of the set $B\cap X_k^{(0)}$ is finite.
It follows that $\mathcal{T}$ is countable.

\begin{lemma}
 \label{lem:cble-fam}
 The family $\mathcal{T}$ is a countable family of open basis of $\partial \aug{Z}$.
\end{lemma}

\begin{proof}
 We have already seen that $\mathcal{T}$ is countable.
 Let $x\in \partial \aug{Z}$, and $n\in \N$. We will show that 
 there exists $U\in \mathcal{T}$ such that 
 \begin{align*}
  x\in U\subset V_n[x]\cap \partial \aug{Z}.
 \end{align*} 
 We choose $t\in \Rp$ such that:
 \begin{align*}
  t&\geq 4E{\Omega^3}n +2,\\
  \rpgm(e,x,t) &= \gamma_e^x(t)\in \Horo(X_k,X_k^{(0)}) \text{ for some } k\in K.
 \end{align*}
 There exists $x'\in \cbHoro(X_k^{(0)})^{(0)}$ such that 
 $d_{\aug{Z}}(\gamma_e^x(t),x')\leq 2$ since $\cbHoro(X_k^{(0)})^{(0)}$ is 2-dense in 
 $\Horo(X_k,X_k^{(0)})$. By \cref{lem:gprod-nearpts}, we have
 \begin{align*}
  (\gamma_e^x(t)\mid x') &\geq 
  \frac{\min\{d_{\aug{Z}}(e,\gamma_e^x(t)),\, d_{\aug{Z}}(e,x')\}}{2Ed_{\aug{Z}}(\gamma_e^x(t),x')}  
  \geq \frac{t-2}{4E}. 
 \end{align*}
 So
 \begin{align*}
  (\gamma_e^x\mid x')&\geq \Omega^{-1}
  \min\{(\gamma_e^x\mid \gamma_e^x(t)),(\gamma_e^x(t)\mid x')\}
  \geq \frac{t-2}{4E\Omega}\geq \Omega^2 n.
 \end{align*}
  Then for all  $y\in W_{t}[x']$, 
  \begin{align*}
   (x\mid y)\geq \Omega^{-2}
   \min\{(x\mid \gamma_e^x),(\gamma_e^x\mid x'),(x'\mid y)\}\geq n.
  \end{align*}
 Here we used that $(x'\mid y)\geq (\gamma_e^{x'}\mid \gamma_e^y)\geq t\geq n$.
 Therefore $x\in W_t[x']\subset V_n[x]\cap \partial \aug{Z}$. 
 This implies that $\mathcal{T}$ is an open basis of $\partial \aug{Z}$.%
\end{proof}

\begin{lemma}
\label{lem:clopen}
 Each $W\in \mathcal{T}$ is a clopen set.
\end{lemma}

\begin{proof}
 Let $k\in K$. Let $x\in \cbHoro(X_k^{(0)})^{(0)}$ and $n\in \N$. 
 We will show that 
 $W_n[x]$ is closed. Let $y\in \partial \aug{Z}\setminus W_n[x]$.
 We will show that there exists $T>0$ such that
 \begin{align*}
  W_{2T}[y]\subset \partial \aug{Z}\setminus W_n[x].
 \end{align*}
 For $p\in \partial \aug{Z}$, set 
 \begin{align*}
   s(k;p):= \inf\{t\in \R: \rpgm(e,p,t)\in X_k^{(0)}\}, \quad
   t(k;p):= \sup\{t\in \R: \rpgm(e,p,t)\in X_k\}
 \end{align*} 
 Here we use the convention that $\inf \emptyset = \infty$
 and $\sup \emptyset = -\infty$.
 First, we suppose $y=c_k\in \overline{\cbHoro(X_k^{(0)})^{(0)}}$.
 Then $s(k;y)<\infty$ and $t(k;y)=\infty$. 
 We remark that the geodesic $\rpgm(\gamma_e^y(s(k;y)),y,-)$ consists of vertical
 line in $\Horo(X_k,X_k^{(0)})$. 
 Let $h$ be the depth of $x$ in $\cbHoro(X_k^{(0)})^{(0)}$. Set
 $T:=h + s(k;y)+ n+D_1$. 
 Let $y'\in W_{2T}[y]$. We have 
 \begin{align*}
  \rpgm(e,y,t) = \rpgm(e,y',t) \qquad (\forall t\in [0,T]).
 \end{align*}
 See \cref{fig:1}. 
 If $y'\in W_n[x]$, then 
 \begin{align*}
  d_Z(\gamma_e^x(n),\rpgm(e,y,n)) =  d_Z(\gamma_e^x(n),\rpgm(e,y',n)) \leq D_1
 \end{align*}
 Thus $(\gamma_x^x \mid \gamma_e^y)\geq n$. This contradicts $y\notin W_n[x]$.
 Therefore $y'\notin W_n[x]$, so we have $W_{2T}[y]\cap W_n[x] = \emptyset$.

 Next, we suppose $y\notin \overline{\cbHoro(X_k^{(0)})^{(0)}}$.
 If $s(k;y) = \infty$, then set $T=2n+D_1$.
 it is easy to see that $W_{2T}[y] \cap W_n[x] = \emptyset$.
 If $s(k;y) <\infty$, then set $T = t(k;y)+2n+D_1$. For $y'\in W_{2T}[y]$, we have
 \begin{align*}
  \rpgm(e,y,t) = \rpgm(e,y',t) \qquad (\forall t\in [0,T]).
 \end{align*}
 See \cref{fig:2}.
 Then by a similar argument as before, $y'\notin W_n[x]$.
 Therefore $W_{2T}[y]\cap W_n[x] = \emptyset$.
\begin{figure}[hb]
\begin{tikzpicture}
\draw(-3.3,0.5)node[right]{$X_k$};
\draw(-4,0)--++(7,0)--++(2,2)--++(-7,0)--cycle;
\draw[very thick, dotted](0,1)--++(0,6)node[above]{$y=c_k$};
\draw[ultra thick](0,1)--++(0,5);
\draw[very thick, dashed](0,1)--++(1,-1);
\draw[ultra thick,](1,0)--++(2,-2)node[right]{$e$};
\draw[very thick, shift={(-0.05,0)}, red](0,1)--++(0,4)--++(-1,0)--++(0,-4);
\draw[very thick, dashed, shift={(-0.05,0)}, red](-1,1)--++(0,-1);
\draw[very thick, shift={(-0.05,0)}, red](-1,0)--++(0,-1);
\draw[very thick, dotted, shift={(-0.05,0)}, red](-1,-1)--++(0,-1)node[below]{$y'$};
\draw[very thick, shift={(-0.05,0)}, red](1,0)--++(2,-2);
\fill(0,3)circle(0.1)node[right]{$\gamma_e^y(s(k;y)+h)$};
\fill(0,1)circle(0.1)node[right]{$\gamma_e^y(s(k;y))$};
\fill(3,-2)circle(0.1);
\end{tikzpicture}
\caption{geodesics $\gamma_e^y$ and $\gamma_e^{y'}$ where $y=c_k$}
\label{fig:1}
\end{figure}
\begin{figure}[t]
\begin{tikzpicture}
\draw(-3.3,0.5)node[right]{$X_k$};
\draw(-4,0)--++(7,0)--++(2,2)--++(-7,0)--cycle;
\draw[ultra thick,](0,1)--++(0,4)--++(-1,0)--++(0,-4);
\draw[ultra thick, dashed](-1,1)--++(0,-1);
\draw[very thick](-1,0)--++(0,-2);
\draw[very thick, dashed](0,1)--++(1,-1);
\draw[very thick,](1,0)--++(2,-2)node[right]{$e$};
\draw[very thick, dotted](-1,-2)--++(0,-1)node[below]{$y$};
\draw[very thick, shift={(-0.05,-0.04)}, red](0,1)--++(0,4)--++(-1,0)--++(0,-4);
\draw[very thick, dashed, shift={(-0.05,0)}, red](-1,1)--++(0,-1);
\draw[very thick, shift={(-0.05,0)}, red](-1,0)--++(0,-1.5)--++(-1,0);
\draw[very thick, dotted, shift={(-0.05,0)}, red](-2,-1.5)--++(-1,0)node[left]{$y'$};
\draw[very thick, shift={(-0.05,0)}, red](1,0)--++(2,-2);
\fill(0,1)circle(0.1)node[right]{$\gamma_e^y(s(k;y))$};
\fill(3,-2)circle(0.1);
\end{tikzpicture}
\caption{geodesics $\gamma_e^y$ and $\gamma_e^{y'}$ where $y\neq c_k$}
\label{fig:2}
\end{figure}

\end{proof}

%
%

\begin{lemma}
\label{lem:c_k-isolated}
 Let $k\in K$. The center $c_k$ of the horoball $\Horo(X_k)$ is an
 isolated point if and only if
 $X_k\cap \xi^{-1}(L^\circ)$ is bounded, where $L^\circ$ denotes a set of
 vertices $l\in L$ which adjacent to at least two different vertices in $K$.
\end{lemma}

\begin{proof}
 For $p\in \partial \aug{Z}$, set 
 \begin{align*}
  t(k;p):= \sup\{t\in \R: \rpgm(e,p,t)\in X_k\}.
 \end{align*}

 First, we suppose that $c_k$ is an accumulation point. Thus there exists a sequence 
 $(x_n)$ in $\partial \aug{Z} \setminus \{c_k\}$ converging to $c_k$. Then we see that
 $t(k;x_n)\to \infty$. We remark that $\rpgm(e,x_n,t(k;x_n))\in X_k\cap \xi^{-1}(L^\circ)$. 
 Therefore $X_k\cap \xi^{-1}(L^\circ)$ is unbounded.

 Next, we suppose that $X_k\cap \xi^{-1}(L^\circ)$ is unbounded. Then there exists an unbounded
 sequence $(v_n)$ in $X_k\cap \xi^{-1}(L^\circ)$. By the definition of $L^\circ$, for each $v_n$, 
 there exists $k_n\in K\setminus \{k\}$ which adjacent to $\xi(v_n)$. Then we see that
 $c_{k_n}$ converges to $c_k$.
 See \cref{fig:3}\begin{figure}[b]
\begin{tikzpicture}
\draw(-7.3,0.5)node[right]{$X_k$};
\draw(-8,0)--++(10,0)--++(2,2)--++(-10,0)--cycle;
\draw[very thick, dotted](0,1)--++(0,6)node[above]{$c_k$};
\draw[ultra thick](0,1)--++(0,5);
\draw[very thick, dashed](0,1)--++(1,-1);
\draw[ultra thick,](1,0)--++(2,-2)node[right]{$e$};
\draw[very thick, shift={(-0.05,0)}, red](1,0)--++(2,-2);
\draw[very thick, shift={(-0.05,-0.04)}, red](0,1)--++(0,2)--++(-1,0)--++(0,-2)node[left]{$v_1$};
\fill[shift={(-0.05,-0.04)},red](-1,1)circle(0.1);
\draw[very thick, dotted, shift={(-0.05,-0.04)}, red](-1,-2)--++(0,-1)node[below]{$c_{k_1}$};
\draw[very thick, dashed, shift={(-0.05,0)}, red](-1,1)--++(0,-1);
\draw[very thick, shift={(-0.05,0)}, red](-1,0)--++(0,-2);
\draw[very thick, shift={(-0.05,-0.04)}, red](0,1)--++(0,3)--++(-2,0)--++(0,-3)node[left]{$v_2$};
\fill[shift={(-0.05,-0.04)},red](-2,1)circle(0.1);
\draw[very thick, dotted, shift={(-0.05,-0.04)}, red](-2,-2)--++(0,-1)node[below]{$c_{k_2}$};
\draw[very thick, dashed, shift={(-0.05,0)}, red](-2,1)--++(0,-1);
\draw[very thick, shift={(-0.05,0)}, red](-2,0)--++(0,-2);
\draw[very thick, shift={(-0.05,-0.04)}, red](0,1)--++(0,4)--++(-3,0)--++(0,-4)node[left]{$v_3$};
\fill[shift={(-0.05,-0.04)},red](-3,1)circle(0.1);
\draw[very thick, dotted, shift={(-0.05,-0.04)}, red](-3,-2)--++(0,-1)node[below]{$c_{k_3}$};
\draw[very thick, dashed, shift={(-0.05,0)}, red](-3,1)--++(0,-1);
\draw[very thick, shift={(-0.05,0)}, red](-3,0)--++(0,-2);
\draw[very thick, dashed, red](-4,-2)--++(-1,0);

\fill(3,-2)circle(0.1);
\end{tikzpicture}
\caption{$c_{k_n}$ converges to $c_k$}
\label{fig:3}
\end{figure}

\end{proof}


\begin{theorem}
\label{thm:Cantor}
 Let $Z$ be a tree of geodesic coarsely convex spaces, 
 and let $\aug{Z}$ be the augmented space of $Z$.
 \begin{enumerate}[label=(\arabic*)]
   \item \label{item:isolate_points}
          Let $\partial \aug{Z}_{\mathsf{iso}}$ 
          denote the set of isolated points in $\partial \aug{Z}$.
        \begin{align*}
         \partial \aug{Z}_{\mathsf{iso}}
         =  \{c_k:k\in K,\ X_k\cap \xi^{-1}(L^\circ)\text{ is bounded }\}
        \end{align*}
  \item \label{item:no-isopoints-Cantor}
        If $\partial \aug{Z}$ has no isolated point, then $\partial \aug{Z}$ 
        is the Cantor space.
 \end{enumerate}
\end{theorem}

\begin{proof}
 By \cref{lem:c_k-isolated}, 
 \begin{align*}
   \partial \aug{Z}_{\mathsf{iso}}\cap \{c_k:k\in K\} =
  \{c_k:k\in K,\ X_k\cap \xi^{-1}(L)\text{ is bounded }\}.
 \end{align*} 
 Let $x\in  \partial \aug{Z}_{\mathsf{iso}}\setminus \{c_k:k\in K\}$.
 Then there exists a sequence $\{t_n\}\subset  \Rp$ and $\{k_n\}\subset K$ such that 
 $\rpgm(e,x,t_n)\in X_{k_n}$ for all $n\in \N$, and $k_n\neq k_m$ for all $n\neq m$.
 Then a sequence $\{c_k\}$ converges to $x$. Thus \ref{item:isolate_points} holds.

 By~\cref{lem:cble-fam,lem:clopen}, $\partial \aug{Z}$ has  
 countable bases consisting of clopen sets. Thus, if $\partial \aug{Z}$ has no isolated points,
 $\partial \aug{Z}$ is homeomorphic to the Cantor space by~\cite{Brouwer-Cantorsp}.
\end{proof}


\begin{remark}
 It is known by {\cite[Lemma 4.6(2)]{bigdely2023relative}} that
 the augmented space $\aug{Z}$ is a Gromov hyperbolic space.
 The ideal boundary of $\aug{Z}$ constructed by the geodesic bicombing
 coincides with the Gromov boundary of $\aug{Z}$ as a 
 Gromov hyperbolic space.
\end{remark}

\section{Topological dimension of ideal boundaries}
\label{sec:topol-dimens-ideal}
In this section, we study the topological dimensions of ideal boundaries
of trees of spaces. Then we compare the results with 
a well-known formula for the cohomological
dimension of free products of groups.
\subsection{Topological dimension}
\label{sec:main-result-tdim}
\begin{theorem}
 \label{thm:dim_delZ}
 Let $(Z,d_Z)$ be a $(\xi,K,L,\{\gamma_k\}_{k\in K})$-tree of geodesic coarsely convex space. 
 Suppose that $Z$ is proper.
 For $k\in K$, set $X_k:=\xi^{-1}(\starNb (k))$.
 Then,
 \begin{align*}
  \dim(\partial Z) = \sup\{\dim (\partial X_k):k\in K\}.
 \end{align*}
\end{theorem}

\cref{thm:dim_free-pd} is obtained by applying \cref{thm:dim_delZ}
to free products of metric spaces.

Let $(Z,d_Z)$ be a $(\xi,K,L,\{\gamma_k\}_{k\in K})$-tree 
of geodesic coarsely convex space.
We fix a bijection $k\colon \N\to K$.
Let $\aug{Z}$ be the augmented space of $Z$.

For simplicity, we denote by 
$\gamma$ an augmented bicombing on $\augn{Z}{n}$
associated with the family $\{\gamma_k\}_{k\in K}$.

\begin{lemma}
 \label{lem:upper_estimate}
 Let $\augn{Z}{n}$ be the $n$-th augmented space as above.
 \begin{align*} 
  \dim \partial \augn{Z}{n}\leq \max\{\dim \partial X_{k(i)}:
  1\leq i\leq n\}.
 \end{align*}
\end{lemma}

\begin{proof}
 We suppose that $N:= \max\{\dim \partial X_{k(i)}:1\leq i\leq n\}<\infty$.
 Let $\mathcal{U}=\{U_\alpha\}_{\alpha \in \Lambda}$ 
 be a finite open cover of $\partial \augn{Z}{n}$. We can suppose
 without loss of generality that $\mathcal{U}$ consists of 
 fundamental neighbourhoods, that is, for $\alpha\in \Lambda$,
 there exists $(x_\alpha,n_\alpha)\in \augn{Z}{n}\cup \N$ such that
 \begin{align*}
  U_\alpha=V_{n_\alpha}[x_\alpha] \cap \partial \augn{Z}{n}
  =\{y\in \partial \augn{Z}{n}: (x\mid y)>n_\alpha\}.
 \end{align*}
 Furthermore, since $\augn{Z}{n}$ is normal, and each $\partial X_{k(i)}$
 is closed,
 we can assume that for each $\alpha\in \Lambda$,
 there exists at most one $i\in \{1,\dots,n\}$ such that
 $U_\alpha\cap \partial X_{k(i)} \neq \emptyset$.
 Set 
 \begin{align*}
  \Lambda_i&:=\{\alpha:U_\alpha\cap \partial X_{k(i)} \neq \emptyset\}\\
  W_i&:=\bigcup_{\alpha\in \Lambda_i}U_\alpha.
 \end{align*}
 We can assume that for $i\neq i'$ and for $\alpha\in \Lambda_i$ 
 and $\alpha'\in \Lambda_{i'}$, we have 
 \begin{align}
  \label{eq:U-disjoint}
  \overline{U_\alpha} \cap   \overline{U_{\alpha'}} = \emptyset.
 \end{align}
 By taking a common refinement of the cover 
 $\{U_\alpha:\alpha\in \Lambda_i\}$ and 
 $\{V_n[x]:(x,n)\in X_{k(n)}\times \N\}$, we can assume that
 for $\alpha\in \Lambda_i$, we have 
 $(x_\alpha,n_\alpha)\in X_{k(n)}\times \N$.


 Since $\dim \partial X_{k(i)}\leq N$, 
 the cover $\{U_\alpha:\alpha\in \Lambda_i\}$ has 
 a refinement $\mathcal{V}_i$ of order $N+1$.

 Now set 
 \begin{align*}
  &F_i:= \bigcup_{\alpha\in \Lambda_i}U_\alpha 
  = \bigcup_{V\in \mathcal{V}_i}V,
  &Z':= \augn{Z}{n}\setminus \left(\bigcup_{1\leq i\leq n}F_i\right)
 \end{align*}
 Since $F_i$'s are open, $Z'$ is closed, 
 so $Z'$ is a compact metrizable space.
 
 \begin{claim}
  For each $i=1,\dots,n$, $F_i$ is closed. Thus $Z'$ is open.
 \end{claim}
 We will show the claim by a similar argument as in the proof of
 \cref{lem:clopen}.
 
 Let $x\in \partial \augn{Z}{n}\setminus F_i$. We will show
 that there exists $T_x>0$ such that $V_{T_x}[x]\cap F_i = \emptyset$.
 Since $\Lambda_i$ is a finite set, it is enough to show that
 for each $\alpha\in \Lambda_i$, there exists $T_{x,\alpha}>0$
 such that $V_{2T_{x,\alpha}}[x]\cap U_\alpha = \emptyset$.
 We remark that $U_\alpha=V_{n_\alpha}[x_\alpha]$ with 
 $(x_\alpha,n_\alpha)\in X_k\times \N$.

 For $k=k(i)\in K$ with $i\geq n$, let $c_k=c_{k(i)}$ be the center
 of the horoball $\Horo(X_k)$. First, we suppose that $x=c_k$ 
 for some $k\in K$. Set 
 \begin{align*}
  s(k;x):=\inf\{t\in \R:\rpgm(e,x,t)\in X_k^{(0)}\}.
 \end{align*}
 The geodesic $\rpgm(\gamma_e^x(s(k;x)),x,-)$ consists of the 
 vertical line. Set $T_{x,\alpha}:=s(k;x)+4D_1+n_{\alpha}$. 
 Then for $y\in V_{2T_{x,\alpha}}[x]$, we have
 \begin{align*}
  \rpgm(e,x,t) = \rpgm(e,y,t) \quad (\forall t\in [0,T_{x,\alpha}])
 \end{align*}
 If $y\in U_\alpha = V_{n_\alpha}[x_\alpha]$, then
 \begin{align*}
  d_Z(\rpgm(e,x_\alpha,n_\alpha),\rpgm(e,x,n_\alpha)) 
  =  d_Z(\rpgm(e,x_\alpha,n_\alpha),\rpgm(e,y,n_\alpha)) \leq D_1
 \end{align*}
 Thus $(x\mid x_\alpha)_e\geq n_\alpha$. This contradicts 
 $x\notin U_\alpha$.
 Therefore $y\notin U_\alpha$, 
 so we have 
 \begin{align*}
  V_{2T_{x,\alpha}}[x]\cap U_\alpha = \emptyset.
 \end{align*}

 Next, we suppose $x\neq c_k$ for all $k\in K$. Then there exists
 $k\in K$ such that 
 \begin{align*}
  s(k;x):=\inf\{t\in \R:\rpgm(e,x,t)\in X_k\}
 \end{align*}
 is finite, and $d_Z(\rpgm(e,x,s(k;x)),x_\alpha)> 2n_\alpha$.
 Set $T_{x,\alpha}=2(s(k;x)+n_{\alpha}+D_1)$. Then for all 
 $y\in V_{T_{x,\alpha}}[x]$, we have
 \begin{align*}
  \rpgm(e,x,t) = \rpgm(e,y,t) \quad (\forall t\in [0,n_{x,\alpha}])
 \end{align*}
 By the same argument as before, $y\notin U_\alpha$,
 so we have $V_{2T_{x,\alpha}}[x]\cap U_\alpha = \emptyset$.
 This completes a proof of the claim.

 Now by a similar argument as in \cref{sec:ideal-bound-augm},
 we see that $Z'$ has countable bases consisting of clopen sets.
 Thus $Z'$ is zero dimensional. So the open cover
 \begin{align*}
  \{U_\alpha \cap Z': \alpha \notin \Lambda_i, \, 1\leq i\leq n\}
 \end{align*}
 has a refinement $\mathcal{V}_\infty$ of order $1$.
 It follow that $\mathcal{U}$ has a refinement
 $\mathcal{V}\cup \bigcup_{i}\mathcal{V}_i $
 of order $N+1$.
 Thus $\dim \partial \augn{Z}{n} \leq N$.
\end{proof}

\begin{proof}[Proof of \cref{thm:dim_delZ}]
  Since each $\partial X_k$ is a closed subset, we have
 \begin{align*}
  \dim(\partial Z) \geq  \sup\{\dim (\partial X_k):k\in K\}.
 \end{align*}
 
 By a formula for the topological dimensions of projective 
 limits in \cite{MR0643755}, we have
 \begin{align*}
  \dim(\partial Z)\leq \sup_n \dim \partial \augn{Z}{n}.
 \end{align*}
 By~\cref{lem:upper_estimate}, we have
 \begin{align*}
  \sup_n \dim \partial \augn{Z}{n} \leq 
  \sup_n \dim \partial X_{k(n)}.
 \end{align*}
 This completes the proof.
\end{proof}

\subsection{Comparison with the formula for cohomological dimensions of groups}
\label{sec:relat-betw-cohom}
We remark that \cref{thm:dim_free-pd} is an analogue of 
a well-known formula for the cohomological dimensions of free products.
For a group $G$, we denote by $\cd G$ the cohomological dimension 
of $G$.
By \cite[VIII (2.4) Proposition, (a) and (c)]{Brown-cohom-gr}
we have the following.
\begin{theorem}[\cite{Brown-cohom-gr}]
\label{thm:cd-freeprod-grp}
 Let $G$ and $H$ be groups. Let $G*H$ be the free product of $G$ and $H$.
 Then we have 
 \begin{align*}
  \cd (G*H) = \max\{\cd G, \cd H\}.
 \end{align*}
\end{theorem}

The following theorem relates the cohomological dimensions of 
group $G$ acting coarsely convex space $X$ and the topological
dimension of the ideal boundary $\partial X$.

\begin{theorem}[{\cite[Corollary 8.10]{FO-CCH}}]
\label{cor:cohom-dim-FO}
 Let $G$ be a group acting geometrically on a proper coarsely convex space $X$.
 If $G$ admits a finite model for the classifying space $BG$, then
 \begin{align*}
  \cd G = \dim \partial X +1.
 \end{align*}
\end{theorem}

Let $G$ and $H$ be groups acting on geodesic coarsely convex spaces
$X$ and $Y$ respectively. 
Suppose $G$ and $H$ admit finite models for the classifying spaces 
$BG$ and $BH$ respectively. In this case, 
$\dim \partial (X*Y)$ can be computed by \cref{thm:cd-freeprod-grp}
as follows.

By \cref{cor:cohom-dim-FO},
\begin{align*}
 \cd(G) = \dim \partial X +1, \quad \cd(H) = \dim \partial Y +1.
\end{align*}
Since the free product $G*H$ is hyperbolic relative to $\{G,H\}$, 
by \cite[Theorem A.1.]{FO-CCH}, $G*H$ 
admits a finite model for the classifying spaces $B(G*H)$. 
By \cref{ex:freeprod-groups}, $G*H$ acts on $X*Y$ geometrically,
so using \cref{cor:cohom-dim-FO} again,
\begin{align*}
 \cd (G*H) = \dim (\partial(X*Y)) + 1.
\end{align*}
Therefore by \cref{thm:cd-freeprod-grp}, we have
\begin{align*}
 \dim(\partial (X*Y)) = \max\{\dim \partial X, \dim \partial Y\}.
\end{align*}


\bibliographystyle{amsplain} 
\bibliography{math}



\end{document}